\documentclass[a4paper,12pt]{article}

\title{Equivalence classes of Niho bent functions}
\author{Kanat Abdukhalikov\\	
UAE University, PO Box 15551, Al Ain, UAE\\
abdukhalik@uaeu.ac.ae}

\date{ }

\usepackage{amsthm,amsmath,amssymb} 
\usepackage{url} 

\begin{document}

\maketitle

\newcommand{\Fn}{\mbox{$\mathbb{F}_{2^n}$}}                
\newcommand{\Fm}{\mbox{$\mathbb{F}_{2^m}$}}                
\newcommand{\Fnm}{\mbox{$\mathbb{F}_{2^m}^*$}}          
\newcommand{\Fmm}{\mbox{$\mathbb{F}_{2^m}^*$}}          
\newcommand{\Fpm}{\mbox{$\mathbb{F}_{p^m}$}}              
\newcommand{\Fp}{\mbox{$\mathbb{F}_{p}$}}                     
\newcommand{\Ft}{\mbox{$\mathbb{F}_{2}$}}                

\theoremstyle{plain}
\newtheorem{theorem}{Theorem}[section]
\newtheorem{lemma}[theorem]{Lemma}
\newtheorem{corollary}[theorem]{Corollary}
\newtheorem{proposition}[theorem]{Proposition}
\newtheorem{fact}[theorem]{Fact}
\newtheorem{observation}[theorem]{Observation}
\newtheorem{claim}[theorem]{Claim}

\theoremstyle{definition}
\newtheorem{definition}[theorem]{Definition}
\newtheorem{example}[theorem]{Example}
\newtheorem{conjecture}[theorem]{Conjecture}
\newtheorem{open}[theorem]{Open Problem}
\newtheorem{problem}[theorem]{Problem}
\newtheorem{question}[theorem]{Question}

\newtheorem{remark}[theorem]{Remark}
\newtheorem{note}[theorem]{Note}

\renewcommand{\arraystretch}{1.5}

\begin{abstract}
Equivalence classes of Niho bent functions are described for all known types of hyperovals. 
\end{abstract}



\section{Introduction}

Bent functions were introduced by Rothaus \cite{Rothaus76}  and then they  were studied by Dillon \cite{Dillon74}. 
A bent function is a Boolean function with an even number of variables which
achieves the maximum possible distance from affine functions \cite{Carlet2010}.  
They are Hadamard difference sets in elementary abelian 2-groups. 
Bent functions have relations to coding theory, cryptography, sequences, combinatorics and design theory \cite{Ab2009,Ab2015,Carlet2010,CarletMes2016,Ding2016}.

Dillon \cite{Dillon74} introduced bent functions related to partial spreads of $F \times F$, $F=\Fm$.
He constructed bent functions that are constant on the elements of a spread.
Dillon also studied a class of  bent functions that are linear on the
elements of a Desarguesian spread. These functions were thoroughly studied in 
\cite{Buda2012,Carlet2011,Dobbertin2006,Helleseth2012,Leander2006} as Niho bent functions.
In \cite{Ab2017,AbMes2017,AbMes2017-2,Ces2015,MesCM2015} these investigations were extended to
other types of spreads, and bent functions which are affine on the elements of spreads, were studied.

Carlet and Mesnager showed  \cite{Carlet2011} that any bent function which is linear on the elements 
of a Desarguesian spread (they are equivalent to Niho bent functions in a bivariate form) 
determines an o-polynomial (oval polynomial).  
Every o-polynomial defines a hyperoval, therefore Carlet and Mesnager revealed 
a general connection between Niho bent functions and hyperovals in Desarguesian planes of even order.  

However there are several equivalence classes of Niho bent functions associated with a fixed hyperoval. 
In \cite{Buda2015,Buda2016} there were constructed several new Niho bent functions 
associated with some o-polynomials and hyperovals.   
In this paper we address the question of finding all equivalence classes of Niho bent functions 
corresponding to a hyperoval. 
This question was mentioned as Open Problem 6 in \cite{Carlet2014}. 
We describe equivalence classes of Niho bent functions for all known types of hyperovals. 
We show that the hyperconics, irregular translation hyperovals, Segre and Glynn hyperovals have respectively 
2, 3, 4 and 4 equivalence classes of Niho bent functions (excluding some exceptional cases in small dimensions) 
and describe these classes. 
The Lunelli-Sce hyperoval has one class and the O'Keefe-Penttila hyperoval has 12 classes. 
For the Payne, Cherowitzo, Subiaco and Adelaide hyperovals the number of equivalence classes of  associated 
Niho bent functions increases exponentially as the dimension of the underlying vector space grows. 
Note that hyperovals are not classified yet, and the list of known hyperovals can be found in  
\cite{Cher,Cher1988,Cher1996}. 


In \cite{Ab2017,Ab2019} it was shown that bent functions linear 
on the elements of a Desarguesian spread  are in one-to-one correspondence with line ovals in an affine plane. 
Points of the line oval completely define the dual bent function. 
More precisely, the zeros of the dual function of a Niho bent function are exactly the points of the line oval 
(in other words, the dual function of a Niho bent function is obtained from the characteristic
function of the set of points of the line oval by adding all-one constant function). 
Hence this gives a general formula \cite{Ab2017,Ab2019} for the dual function for any Niho bent function 
(this question was mentioned as Open Problem 10 in  \cite{Carlet2014}).  
Furthermore, Niho bent functions are in one-to-one correspondence with ovals with nucleus at a designated point. 
Therefore, we have geometric characterization of Niho bent functions and of their duals. 

In \cite{Ab2017,Ab2019}  special $g$-functions were introduced as a new analog  of o-polynomials in case of 
univariate  presentation of functions and hyperovals. Using $g$-functions allow us to implement methods 
which do not employ the ``unusual magic action" defined by O'Keefe and Penttila \cite{OKeefe2002}. 
A criteria for existence of $g$-functions is presented in \cite{Ab2019}.

We note that there are no analogs of such bent functions in case of odd characteristic: 
 \c{C}e\c{s}melio\u{g}lu, Meidl and  Pott  \cite{Ces2015} showed that  
bent functions which are affine on elements of Desarguesian spreads over a field of odd characteristic 
will be constant  on the elements of the spread. 

The paper is organized as follows.
We recall first in Section \ref{preliminary} definitions and notation concerning bent functions, ovals and line ovals. 
In Section \ref{g-function} we study connections between $g$-functions and  o-polynomials, 
and establish general formulas for all $g$-functions and Niho bent functions corresponding to a fixed hyperoval. 
Theorem \ref{thm:oval} gives a formula for calculation of the $g$-function corresponding to a given  o-polynomial.  
In Section \ref{Niho-ovals} we study equivalence classes of Niho bent functions corresponding to all known 
hyperovals. 
Finally, in dimensions up to $m=6$ the equivalence classes were considered in detail.


\section{Preliminary considerations and notation}
\label{preliminary}

In this section we recall some definitions and notation.

\subsection{Bent functions}

Let $K=\mathbb{F}_{2^n}$ and $\mathbb{F}_{2}$ be finite fields of orders
$2^n$ and $2$ respectively.
Let $\mathbb{F}^n_2$ be an $\mathbb F_2$-vector space of dimension $n$. We shall endow
$\mathbb{F}^n_2$ with the  structure of the field $\mathbb{F}_{2^n}$.
A Boolean function on $\mathbb{F}_{2^n}$  is a mapping from $\mathbb{F}_{2^n}$ to the prime field $\mathbb{F}_2$.

If $f$ is a Boolean function defined on $\mathbb{F}_{2^n}$, then the Walsh transform of $f$   is defined as follows:
\begin{equation}
\label{Walsh}
W_f(b) = \sum_{x\in \mathbb{F}_{2^n}} (-1)^{f(x)+ b \cdot x}, 
\end{equation}
where $b\cdot x$ is  a scalar product  from  $\mathbb{F}_{2^n}\times \mathbb{F}_{2^n}$ to $\mathbb{F}_2$.
A Boolean function $f$  on $\mathbb{F}_{2^n}$ is said to be {\em bent} if its Walsh transform satisfies $W_f(b) = \pm 2^{n/2}$ for all $b \in \mathbb{F}_{2^n}$. Then $n$ is an even integer. 

Given a bent function $f$ over $\mathbb{F}_{2^n}$, one can define its {\em dual function}, denoted by
$\tilde{f}$,  by considering the signs
of the values of the Walsh transform $W_f(b)$  of $f$. More precisely, $\tilde{f}$ is defined by the equation:
$$(-1)^{ \tilde{f}(x)}2^{n/2}=W_f(x). $$
The dual of a bent function is bent again, and  $\tilde{\tilde{f}}=f$.

\subsection{Polar representations}

Let $F=\mathbb{F}_{2^m}$ be a finite field of order $q=2^m$. Consider $F$ as a subfield of $K=\mathbb{F}_{2^n}$, 
where $n=2m$, so $K$ is a two dimensional vector space over $F$. Let $F_0= \mathbb{F}_2$ be a prime field. 

As usually, $Tr_{M/L}(x)$ is the trace function with respect to a finite field extension $M/L$. 
For particular field extensions we denote the corresponding trace functions by 
$$Tr (x)=Tr_{K/ F_0}(x), \quad T (x)=Tr_{K / F}(x), \quad tr(x)=Tr_{F/ F_0}(x).$$
The \emph{conjugate} of $x\in K$ over $F$ is
$$\bar{x} =x^q.$$
Then the \emph{trace}  and the \emph{norm} maps from $K$ to $F$ are 
$$T(x) = Tr_{K/F}= x + \bar{x}= x + x^q,$$ 
$$N(x) = N_{K/F} (x) = x \bar{x} = x^{1+q}.$$
The \emph{unit circle} of $K$ is the set of elements of norm $1$:
$$S= \{ u\in K : u\bar{u} =1 \}. $$
Therefore, $S$ is the multiplicative group of $(q+1)$st roots of unity in $K$.
Since $F\cap S = \{ 1\}$, each non-zero element of $K$ has  a unique polar coordinate representation
 $$ x=\lambda u$$
 with $\lambda\in F^*$ and $u \in S$. For any $x\in K^*$ we have
 $$\lambda = \sqrt{x \bar{x}},$$
 $$u= \sqrt{x/ \bar{x}}.$$
 
 One can define nondegenerate bilinear form $\langle \cdot , \cdot \rangle : K \times K \rightarrow F$ by
$$\langle x,y \rangle =T(x\bar{y})=x\bar{y}+\bar{x}y.$$
Then the form $\langle \cdot , \cdot \rangle$ is alternating and symmetric: 
$$\langle a,a\rangle =0,$$
$$\langle a,b \rangle =\langle b,a\rangle .$$


\subsection{Affine and projective planes}

Consider points of a projective plane $PG(2,q)$ in homogeneous coordinates as triples $(x : y : z)$, 
where $x, y, z \in F$, $(x,y,z)\neq (0,0,0)$, 
and we identify $(x : y : z)$ with $(\lambda x : \lambda y : \lambda z)$, $\lambda\in F^*$. Then 
points of $PG(2,q)$ are 
$$\{ (x : y : 1) \mid x \in F, \ y \in F \} \cup \{ (x:1:0) \mid x\in F\}  \cup \{ (1:0:0) \},$$

For $a, b, c \in F$, $(a,b,c)\neq (0,0,0)$,  the line $[a:b:c]$ in $PG(2,q)$ is defined  as 
$$[a:b:c] = \{ (x:y:z) \in PG(2,q) \mid ax+by+cz =0\}.$$
Triples $[a:b:c]$ and $[\lambda a: \lambda b: \lambda c] $ with $\lambda\in F^*$ define same lines. 
Then any line in $PG(2,q)$ is one of the following lines: 
\begin{eqnarray*}
[a:b:1] & = & \{ (x:y:1) \mid ax+by+1=0\} \cup \{(b:a:0)\}, \ (a,b) \neq (0,0), \\
{[ 0:0:1]}          & = & \{ (x:1:0) \mid x\in F \} \cup  \{(1:0:0)\},  \\
{[a:1:0]}           & = & \{ (x:ax:1) \mid x\in F \} \cup  \{(1:a:0)\},  \ a\in F, \\
{[1:0:0]}          & = & \{ (0:y:1) \mid y\in F \} \cup  \{(0:1:0)\}.  
\end{eqnarray*}

The point $(x : y : z)$ is incident with the line $[a:b:c]$ if and only if $ax+by+cz=0$. 
The map $(x : y : z) \mapsto [x:y:z]$ defines a duality \cite{Hir} for $PG(2,q)$. 

We shall call points of the form $(x:y:0)$  the points at infinity. 
Then $[0:0:1]$ indicates the line at infinity (it consists of all points at infinity). 
We define an affine plane $AG(2,q) = PG(2,q)\setminus [0:0:1]$, so points of this affine plane $AG(2,q)$ are 
$$\{ (x : y : 1) \mid x, y \in F \} .$$
Associating $(x : y : 1)$ with $(x,y)$ we can identify points of the affine plane $AG(2,q)$ with  elements 
of the vector space  
$$V(2,q) = \{ (x,y) \mid x, y \in F \} ,$$
and we will write $AG(2,q)=V(2,q)$. 

Lines in $AG(2,q)=V(2,q)$ are  $\{ (c,y) \mid y \in F\}$ and $\{ (x,xb +a) \mid x \in F\}$, $a, b, c \in F$.
These lines can be described by equations $x=c$ and $y=x b +a$.

We introduce now other representation of $PG(2,q)$ using the field $K$. 
Consider pairs $(x : z)$, where $x \in K$, $z\in F$, $x\neq 0$ or 
$z\neq 0$, and we identify $(x : z)$ with $(\lambda x : \lambda z)$, $\lambda\in F^*$. 
Then points of $PG(2,q)$ are 
$$\{ (x : 1) \mid \ x \in K \} \cup \{ (u:0) \mid u\in S\}.$$

For $\alpha \in K$ and $\beta\in F$ we define lines $[\alpha :\beta ]$ in $PG(2,q)$  as 
$$[\alpha :\beta ] = \{ (x:z) \in PG(2,q) \mid \langle \alpha,x \rangle +\beta z =0\}.$$
Pairs $[\alpha :\beta ]$ and $[\lambda\alpha :\lambda\beta ]$ with $\lambda\in F^*$ define same lines. 
Then any line in $PG(2,q)$ is one of the following lines: 
\begin{eqnarray*}
[\alpha :1] & = & \{ (x:1) \mid \langle \alpha,x \rangle +1=0\} \cup \{(\alpha :0)\}, \ \alpha\neq 0, \\
{[ 0:1]}          & = & \{ (u:0) \mid u \in S \}, \\
{[u:0]}           & = & \{ (\lambda u:1) \mid \lambda \in F\} \cup \{ (u:0) \}, \ u\in S. 
\end{eqnarray*} 

The point $(x : z)$ is incident with the line $[\alpha : \beta]$ if and only if $\langle \alpha,x \rangle +\beta z =0$. 
The map $(x : z) \mapsto [x:z]$ defines a duality \cite{Hir} for $PG(2,q)$. 

Element $u_{\infty} = (u:0)$, $u\in S$,  will be referred to as the point at infinity in the direction of $u$. 
So $[0:1]$ indicates the line at infinity.  

We define an affine plane $AG(2,q) = PG(2,q)\setminus [0:1]$, so points of this affine plane $AG(2,q)$ are 
$$\{ (x : 1) \mid x \in K \} .$$

Associating $(x : 1)$ with $x \in K$ we can identify points of the affine plane $AG(2,q)$ with elements of 
the field $K$, and we will write $AG(2,q)=K$. 

Lines of  $AG(2,q)=K$ can be considered as the zeroes of an equation $\langle a,x \rangle+b=0$. 
Normalizing  $a$ to $u\in S$, we see that lines of  $AG(2,q)$ can be considered  as the zeroes of an equation 
$\langle u,x \rangle+\mu =0$:
$$L(u,\mu) =\{ x \in K : \langle u,x \rangle+\mu =0 \},$$
where $u\in S$ and $\mu\in F$ (see for details  \cite[subsection 2.1]{Ball1999}).
$L(u,\mu)$ can be considered as a line in the direction of  $u$. 
Note that there are $(q+1)q =q^2+q$ such lines. 
Lines $L(u,\lambda)$ and $L(u,\mu)$ are parallel, and $L(u,0)=\{ \lambda u \mid \lambda \in F \}$. 
Lines $L(u,\lambda)$ and $L(v,\mu)$ are not parallel for distinct $u, v \in S$, since $\langle x,y\rangle$ is 
a nondegenerate alternating bilinear form on $K$, considered as a two dimensional vector space over $F$.

Throughout the paper, we will consider these two  representations of projective planes $PG(2,q)$, 
and for each of such projective planes we consider a fixed affine plane $AG(2,q)$, described above. 
They will be written as $AG(2,q)=V(2,q)$ and $AG(2,q)=K$.

\subsection{Ovals and line ovals}

Let $PG(2,q)$ be a finite projective plane of order $q=2^m$.
An \emph{oval} is a set of $q+1$ points, no three of which are collinear. Dually, a \emph{line oval} is a set of $q+1$
lines no three of which are concurrent. Any line of the plane meets the oval $\mathcal{O}$ at either
0, 1 or 2 points and is called exterior, tangent or secant, respectively. 
All the tangent lines to the oval $\mathcal{O}$ concur  \cite{Hir} at the same point $N$, called the \emph{nucleus} 
 of $\mathcal{O}$. The set $\mathcal{H}=\mathcal{O} \cup N$ becomes a \emph{hyperoval}, that is a set of $q+2$
points, no three of which are collinear. 
Conversely, by removing any point from hyperoval one gets an oval. 
If $\mathcal{O}'$ is a line oval, then there is exactly one line $\ell$
such that on each of its points there is only one line of $\mathcal{O}'$. This line is called the (\emph{dual})
\emph{nucleus} of $\mathcal{O}'$. 
The $(q+2)$-set $\mathcal{H}'=\mathcal{O}'\cup \{ \ell\}$ is a \emph{line hyperoval} or \emph{dual hyperoval}.

By a \emph{line oval} $\mathcal{O}$ in an affine plane $AG(2,q)$ we assume a set of $q+1$ nonparallel 
lines in $AG(2,q)$, such that these lines, extended by the corresponding points at infinity, 
determine a line oval in $PG(2,q)$ whose nucleus is the line at infinity. 

For any oval in $PG(2,q)$ there are $q(q+1)/2$ secants and $q(q-1)/2$ exterior lines. 
Let $\mathcal{O}$ be a line oval in the affine plane $AG(2,q)=V(2,q)$ 
and $E(\mathcal{O})$ the set of points which are on the lines of the line oval $\mathcal{O}$:
$$E(\mathcal{O}) = \{ (x,y) \in V(2,q) \mid (x,y) \ \rm{ is \ on \ a \ line \ of \ } \mathcal{O} \}.$$
Then each point of $E(\mathcal{O})$ belongs  to two lines of $\mathcal{O}$, 
$$|E(\mathcal{O})|=q(q+1)/2 $$ 
and there are $q(q-1)/2$ points in $AG(2,q)$ that do not belong to $\mathcal{O}$ 
(see, for example, \cite{Dem,Hir}).  
Note that Kantor \cite[Theorem 7]{Kantor75} showed that $E(\mathcal{O})$ is a difference set.

Any line oval in $AG(2,q)=K$ can be represented \cite{Ab2019} as $\mathcal{O} = \{ L(u,g(u)) : u\in S\}$ 
for some function $g: S \rightarrow F$.

\subsection{Niho bent functions} 
\label{bent-1var}

Desarguesian spread on $V(2,q)$ is a collection  of $q+1$  one-dimensional subspaces 
 $\{ (x,xt) \mid x\in F \}$, $t\in F$, and $\{(0,y) \mid y\in F \}$.
So every nonzero element of $V(2,q)$ lies in a unique subspace \cite{Dem}. 

The field $K$ can be considered as a two dimensional vector space over $F$. 
Then the set
$$\{ uF : u \in S \}$$
is a spread on this vector space.  We consider Boolean functions $f: K \rightarrow \Ft$, which are $\Ft$-linear 
on each element of the spread $\{ uF : u \in S \}$.
Then  for any $x =\lambda u\in K$, where $\lambda \in F$, $u\in S$, the function $f$ can be defined by
\begin{equation}
\label{func}
f(x)=tr(\lambda g(u))
\end{equation}
for some function $g: S \rightarrow F$. 
If such a function is bent then it is called  {\em Niho bent function}. These functions can be written with the help 
of Niho power functions, so is the name.

Let Boolean function $f: K \rightarrow \Ft$ be linear on elements of the spread $\{ uF : u \in S \}$. 
Now we consider bent functions, and we use the scalar product $a \cdot b = tr(\langle a,b\rangle )$ 
in the formula (\ref{Walsh}) for the Walsh transform (as it was introduced in \cite{Ab2019}). 
The following result was proved in \cite{Ab2019}: 
 \begin{theorem}
\label{thm:main}
Let function $f : K \rightarrow {\mathbb F}_2$ be defined by  $f(\lambda u)=tr(\lambda g(u))$,  
for some function $g: S \rightarrow F$. 
Then the following statements are equivalent:
\begin{enumerate}
\item
The function $f$  is  bent; 
\item
The set $\{ L(u,g(u)) \mid u\in S\}$  is a line oval in $K=AG(2,q)$; 
\item
The set $\{ \frac{u}{g(u}) \mid u\in S\}$ is  an oval in $PG(2,q)$ whose nucleus is the origin. 
\end{enumerate}
\end{theorem}

Here we assume that if $g(u)=0$ then $\frac{u}{g(u)} = u_{\infty}$ is an element at infinity in the direction $u$. 

Therefore, Niho bent functions are in one-to-one correspondence with the ovals in  $PG(2,q)$ with nucleus 
at the origin. A function $g: S \rightarrow F$ is said to be a  $g$-function if it satisfies conditions of 
Theorem \ref{thm:main}. 

Boolean functions $f, f': \mathbb{F}_{2^n}\rightarrow \mathbb{F}_2$ are \emph{extended-affine equivalent}
 if there exist an affine permutation $L$ of $\mathbb{F}_{2^n}$ and an affine function 
$\ell:\mathbb{F}_{2^n}\rightarrow \mathbb{F}_2$ such that $f'(x)=(f\circ L)(x)+\ell(x)$. 
In our paper we will call such functions just \emph{equivalent}. 
If Boolean functions $f$ and $f'$ are equivalent and $f$ is bent then $f'$ is bent too. 

We recall that the automorphism (collineation) group of $PG(2,q)$ is
$P\Gamma L(3,q) = PGL(3,q)\langle\sigma\rangle$ and the automorphism group of $AG(2,q)$ is
$A\Gamma L(2,q) = AGL(2,q)\langle\sigma\rangle$, where $\langle\sigma\rangle$ is the Galois group of $F$,
and $AGL(2,q) = F^2\cdot GL(2,q)$ is the affine group. 
Hyperovals (and ovals) are called (projectively) equivalent if they are  equivalent under the action of the group 
$P\Gamma L(3,q)$.  

Niho bent functions are equivalent if and only if the corresponding ovals (with nucleus at the origin) 
are projectively equivalent \cite{Ab2017,Penttila2014}. 
Note also that functions $g(u)$ and $g(u) + \langle c, u\rangle$, where $c\in K$, lead to equivalent ovals 
\cite{Ab2019}.

\begin{remark}
In thesis \cite{Deorsey2015}, following ideas from \cite{Fisher2006}, the $\rho$-polynomials were introduced. 
We note that the $\rho$-polynomials and our $g$-functions are connected in the following way: $g(u)=1/\rho(u)$.
\end{remark}

\section{O-polynomials, $g$-functions and Niho bent functions} 
\label{g-function}

Following \cite{Fisher2006}, consider an element $i\in K$ with property $T(i)=i+i^q=1$. 
Then $K=F(i)$ and $i$  is a root of a quadratic equation 
$$z^2+z+\delta=0,$$ 
where $\delta = N(i) \in F$. 
Any element $z\in K$ can be represented as $z=x+yi$,  where $x, y \in F$. 
For $z=x+yi$ we have $x= \langle i, z\rangle$ and $y= \langle 1, z\rangle$. 

Note that if $m$ is odd then one can choose  $i=\omega$, $\omega^2 +\omega +1=0$.   
So if $w\in K$ is a generator of $S$ then we can take $i =\omega = w^{(q+1)/3}$. 
 
Every hyperoval is equivalent to one, given by  
$$\mathcal{D} (h)= \{( t: h(t): 1) \mid t\in F\} \cup \{ (1:0:0) \} \cup \{ (0:1:0) \},$$
where  $h(t)$ is an o-polynomial. 
We can describe ovals in a similar way, but ensuring that the nucleus of the oval is the point $(1:0:0)$: 
$$\mathcal{E} (h)= \{( t: h(t): 1) \mid t\in F\}  \cup \{ (0:1:0) \}.$$
Now we would like to find an equivalent representation of this  oval in $PG(2,q)$ using the field $K$,  
namely, we would like to write the oval $\mathcal{E} (h)$ in the form 
$$\left\{ \frac{u}{g(u)} \mid u \in S\right\}$$ 
for some function $g: S \rightarrow F$. 


\begin{theorem}
\label{thm:oval} 
Let $h(t)$ be an  o-polynomial. 
Define an oval $\mathcal{E}(h)$ with nucleus $(1:0:0)$ by 
$$\mathcal{E}(h) = \{( t: h(t): 1) \mid t\in F\} \cup \{ (0:1:0) \} .$$
Then corresponding function $g(u)$  can be determined by 
$$g(u) = h^{-1} \left(\frac{ \langle i, u \rangle}{ \langle 1, u\rangle} \right) \langle 1, u\rangle+  \langle i, u\rangle, 
\quad g(1)=1.$$ 
\end{theorem}

\begin{proof} 
We apply a collineation to the hyperoval $\mathcal{D}(h)$ in order to get a hyperoval without points on infinity, 
such that the point $(1:0:0)$ is mapped to $(0:0:1)$ and the point $(0:1:0)$ is mapped to $(1:0:1)$. 
Define 
$$\alpha ( (a:b:c) ) = (b:c:(a+b+cd))$$ 
for some $d\in F$. 
Then $\alpha ((1:0:0))  = (0:0:1)$, $\alpha ((0:1:0)) = (1:0:1)$ and  $\alpha ((h^{-1}(t):t:1)) = (t:1:(t+h^{-1}(t)+d))$. 
We want to have $t+h^{-1}(t)+d\ne 0$ for all $t\in F$. It happens if we choose $d\in F$ in such a way that the line 
$x+y+dz=0$ does not intersect the hyperoval  
$\mathcal{H}' =  \{( t: h^{-1}(t): 1) \mid t\in F\} \cup \{ (1:0:0) \} \cup \{ (0:1:0) \}$. 

Now we associate the point $(x:y:1)$ with the point $z=x+yi \in K$. 
Then $(1:0:1)$ corresponds to $1\in K$ and $(0:0:1)$ corresponds to $0\in K$. 
Assume that the oval $\alpha(\mathcal{E}(h))$ corresponds to the oval $\left\{ \frac{u}{g(u)} \mid u \in S\right\}$. 
If $z= \frac{u}{g(u)} = x+yi$ then $x = \frac{1}{g(u)} \langle i, u\rangle$ and $y = \frac{1}{g(u)} \langle 1, u\rangle$. 
Considering points  
$$\left( \frac{t}{t+h^{-1}(t)+d}:  \frac{1}{t+h^{-1}(t)+d}: 1\right)$$  
from the oval $\alpha(\mathcal{E}(h))$  we see that 
$$ \frac{t}{t+h^{-1}(t)+d} =  \frac{1}{g(u)} \langle i,u\rangle , $$
$$ \frac{1}{t+h^{-1}(t)+d} =  \frac{1}{g(u)} \langle 1,u\rangle . $$
Therefore, 
$$t = \frac{ \langle i, u \rangle}{ \langle 1, u\rangle},$$
$$g(u) = (t+h^{-1}(t)+d ) \cdot \langle 1,u\rangle = 
h^{-1} \left(\frac{ \langle i, u \rangle}{ \langle 1, u\rangle} \right) \langle 1, u\rangle+  \langle i, u\rangle + 
\langle d, u\rangle .$$
One has $g(1) =1$ since  $1 \in \{ \frac{u}{g(u)} \mid u \in S \}$ and the equality $1=\frac{u}{g(u)}$ implies 
$u=1$, $g(u)=1$.   
Finally, we can subtract  from $g(u)$ the linear part $\langle d, u\rangle$, since adding linear function 
$\langle d, u\rangle$ to $g(u)$  will produce an equivalent function. 
\end{proof} 

\begin{remark} 
In place of the function 
$g(u) = h^{-1} \left(\frac{ \langle i, u \rangle}{ \langle 1, u\rangle} \right) \langle 1, u\rangle+  \langle i, u\rangle$ 
one can consider 
$g'(u) = h^{-1} \left(\frac{ \langle i, u \rangle}{ \langle 1, u\rangle} \right) \langle 1, u\rangle$, they produce 
equivalent ovals and equivalent Niho bent functions, but in this case one has $g'(1)= 0$. 
\end{remark}

\begin{corollary}
\label{cor:poly} 
Let $h(t)=  t^s$ be an o-polynomial. Then its corresponding $g$-function and Niho bent function can be written 
respectively as  
$$g(u)=  \langle i, u\rangle^{s^{-1}} \cdot \langle 1, u\rangle^{q-s^{-1}},$$ 
$$f(x)=  tr \big(  \langle i, x\rangle^{s^{-1}} \cdot \langle 1, x\rangle^{q-s^{-1}}  \big) = 
 tr \big(  (\bar{i} x + i \bar{x})^{s^{-1}} \cdot (x + \bar{x})^{q-s^{-1}}  \big),$$
 where $s^{-1}$ is the inverse of  $s$ modulo $q-1$. 
\end{corollary}

\begin{proof} 
 We have 
$$g(u) = h^{-1}\left(\frac{\langle i, u\rangle}{\langle 1, u\rangle} \right) \langle 1, u\rangle = 
\left(\frac{\langle i, u\rangle}{\langle 1, u\rangle} \right)^{s^{-1}} \langle 1, u\rangle =  
\langle i, u\rangle^{s^{-1}} \cdot \langle 1, u\rangle^{q-s^{-1}}.$$
Further, since $x=\lambda u$, $\lambda = \sqrt{x\bar{x}}$, $u=\sqrt{x/\bar{x}}$, we have  
\begin{eqnarray*} 
f(x)  
& =  &  
tr (\lambda g(u))  \\
& =  &  
tr \big(\sqrt{x\bar{x}} \cdot 
\langle i, \sqrt{x/\bar{x}}\rangle^{s^{-1}} \cdot \langle 1, \sqrt{x/\bar{x}}\rangle^{q-s^{-1}} \big)  \\
& =  &  
tr \big(\sqrt{x\bar{x}} \cdot 
\big(\bar{i} \sqrt{x/\bar{x}} + i \sqrt{\bar{x}/x}\big)^{s^{-1}} \cdot  \big(\sqrt{x/\bar{x}} + \sqrt{\bar{x}/x} \big)^{q-s^{-1}} \big) \\
& =  &  
tr \left(\sqrt{x\bar{x}} \cdot 
\left(\frac{\bar{i} x + i \bar{x}}{\sqrt{\bar{x} x}}\right)^{s^{-1}} \cdot  
\left(\frac{\bar{i} x + i \bar{x}}{\sqrt{\bar{x} x}}\right)^{q-s^{-1}} \right) \\
& =  &  
tr \big(  (\bar{i} x + i \bar{x})^{s^{-1}} \cdot (x + \bar{x})^{q-s^{-1}}  \big), 
\end{eqnarray*} 
which completes the proof. 
\end{proof}

\begin{remark} 
In place of the oval $\mathcal{E}(h)$ one can consider the oval 
$\mathcal{E}'(h)= \{( t, h(t), 1) \mid t\in F\} \cup \{ (1,0,0) \} $ with nucleus $(0,1,0)$. 
Then corresponding $g$-function can be written as   
$g(u) = h \left(\frac{ \langle i, u \rangle}{ \langle 1, u\rangle} \right) \langle 1, u\rangle+  \langle i, u\rangle$, 
$g(1)= 1$. 
However this approach is less convenient for our reasonings in  Section  \ref{Niho-ovals}. 
\end{remark}

Next we derive formulas for Niho bent functions and $g$-functions obtained from ovals in $K$.

\begin{theorem}
\label{thm:oval-f} 
Let $\mathcal{O}$ be an oval in $K$ whose  nucleus is the  origin. Then the associated Niho bent function is 
$$ f(x)=\sum_{j=0}^{m-1} \sum_{i=0}^{q}  \left(\sum_{v\in {\mathcal O}} \frac{1}{v^{i(q-1)+2^j}} \right)x^{i(q-1)+2^j}.$$
\end{theorem}

\begin{proof} 
By  \cite[Theorem 3.8]{Ab2017} we have 
$$f(x)= \sum_{v\in {\mathcal O}} \Big[(x^{q^2-q} + v^{q^2-q})^{q^2-1} +1\Big] \sum_{j=0}^{m-1} (x/v)^{2^j}.$$
It is clear  that $f(0)=0$. Let $x\ne 0$.  Then 
\begin{eqnarray*}
f(x)  
& =  & 
\sum_{v\in {\mathcal O}} \Big[\Big(\frac{1}{x^{q-1}} + \frac{1}{v^{q-1}}\Big)^{q^2-1} +1\Big] \sum_{j=0}^{m-1} (x/v)^{2^j} \\
& =  & 
\sum_{v\in {\mathcal O}} \Big[\Big(\frac{x^{q-1} + v^{q-1}}{x^{q-1}v^{q-1}}\Big)^{q^2-1} +1\Big] 
\sum_{j=0}^{m-1} (x/v)^{2^j} \\
& =  & 
\sum_{v\in {\mathcal O}} \Big[\big(x^{q-1} + v^{q-1}\big)^{q^2-1} +1\Big] \sum_{j=0}^{m-1} (x/v)^{2^j} \\
& =  & 
\sum_{v\in {\mathcal O}} \Big[\sum_{i=0}^{q^2-1}((x^{q-1})^i  (v^{q-1})^{q^2-1-i} +1\Big] \sum_{j=0}^{m-1} (x/v)^{2^j}. 
\end{eqnarray*}
We note that some terms of the second sum are the same:
$$(x^{q-1})^{k(q+1)+i} = (x^{q-1})^{i}, \quad (v^{q-1})^{k(q+1)+i} = (v^{q-1})^{i}.$$
Therefore, 
\begin{eqnarray*}
f(x)  
& =  & 
\sum_{v\in {\mathcal O}} \Big[\sum_{i=0}^{q^2-1}((x^{q-1})^i  (v^{q-1})^{q^2-1-i} +1\Big] \sum_{j=0}^{m-1} (x/v)^{2^j} \\
& =  & 
\sum_{v\in {\mathcal O}} \Big[\sum_{i=0}^{q}((x^{q-1})^i  (v^{q-1})^{q+1-i} \Big] \sum_{j=0}^{m-1} (x/v)^{2^j} \\
& =  & 
\sum_{v\in {\mathcal O}} \sum_{i=0}^{q}(((x/v)^{q-1})^i   \sum_{j=0}^{m-1} (x/v)^{2^j} \\
& =  & 
\sum_{j=0}^{m-1} \sum_{i=0}^{q}   \left(\sum_{v\in {\mathcal O}} \frac{1}{v^{i(q-1)+2^j}} \right)x^{i(q-1)+2^j} ,  
\end{eqnarray*} 
which establishes the formula. 
\end{proof} 

\begin{theorem}
\label{thm:g-function} 
Let $\mathcal{O}= \left\{ \frac{u}{g(u)} \mid u \in S\right\}$ be an oval in $K$ whose nucleus is the origin. Then 
$$ g(u)= \sum_{i=0}^{q}  \sum_{v\in {\mathcal O}} v^{(q-1)i/2 -1}u^{i+1}.$$
\end{theorem}

\begin{proof} 
If $v \in \mathcal{O}$ then  $v =z/g(z)$ for some $z\in S$.  Then  we have 
$$v^{(q-1)/2} =  (z/g(z))^{(q-1)/2} =  z^{(q-1)/2}= \bar{z}.$$
Now fix $u \in S$.  If $v \in \mathcal{O}$, $v =z/g(z)$ and $z = u$ then 
$$\sum_{i=0}^q (u v^{(q-1)/2}))^i \ \frac{u}{v} = \sum_{i=0}^q (u \bar{z})^i \ \frac{u}{v} = \frac{u}{v} = g(u).$$
If $v \in \mathcal{O}$, $v =z/g(z)$ and $z \ne u$ then 
$$\sum_{i=0}^q (u v^{(q-1)/2}))^i \ \frac{u}{v} = \sum_{i=0}^q (u \bar{z})^i \ \frac{u}{v} = 0.$$

Therefore, 
\begin{eqnarray*}
g(u)  
& =  & 
\sum_{v\in {\mathcal O}} \sum_{i=0}^q (u v^{(q-1)/2}))^i \ \frac{u}{v} \\
& =  & 
\sum_{i=0}^{q}  \sum_{v\in {\mathcal O}} v^{(q-1)i/2 -1}u^{i+1},  
\end{eqnarray*} 
which proves the theorem. 
\end{proof} 

Given a hyperoval in a projective plane $PG(2,q)$, an oval can be obtained by deleting one of the points 
of the hyperoval. 
This deleted point is the nucleus of the resulting oval. There are $q+2$ ovals which can be obtained in this way, 
but some of them will be equivalent under the action of the automorphism group $P\Gamma L(3,q)$ of  $PG(2,q)$. 
If $P$ and $Q$ are points of a hyperoval $\mathcal{H}$, then ovals 
$\mathcal{H} \setminus \{ P\}$ and $\mathcal{H} \setminus \{ Q\}$ are equivalent if and only if 
$P$ and $Q$ lie in the same orbit of the stabilizer of $\mathcal{H}$ on $\mathcal{H}$. 
(The stabilizer of $\mathcal{H}$ in $P\Gamma L(3,q)$ is also called automorphism group of $\mathcal{H}$.) 
Therefore, the number of projectively inequivalent ovals obtained from hyperoval ${\mathcal H}$  is equal 
to the number of orbits of ${\mathcal H}$ under the action of the stabilizer of ${\mathcal H}$.

Assume that $g(u) \ne 0$ for all $u\in S$.  Function $g(u)$ determines an hyperoval  
$$\mathcal{H}= \left\{ \frac{u}{g(u)} \mid u \in S\right\}\cup \{ 0\}$$ 
in $K$. 
Consider an oval $\mathcal{H} \setminus \{ s/g(s)\}$, $s\in S$, whose nucleus is  $s/g(s)$, and shift it 
by the element $s/g(s)$ in order to get an oval $\mathcal{O}_s$, whose nucleus is the origin: 
$$\mathcal{O}_s = \left\{ \frac{v}{g(v)} + \frac{s}{g(s)} \mid v \in S, v \ne s \right\} \cup \left\{ \frac{s}{g(s)} \right\}.$$
Any Niho bent function associated with the hyperoval $\mathcal{H}$ is equivalent to a Niho bent function obtained 
from one of the ovals $\mathcal{O}_s$,  and two such bent functions $f_s(x)$ and $f_t(x)$ are equivalent 
\cite{Ab2017,Penttila2014}  if and only if the points 
$s/g(s)$ and $t/g(t)$ are in the same orbit under action of automorphism group of the hyperoval $\mathcal{H}$.   

\begin{theorem}
\label{thm:g-neighbor} 
Let $\mathcal{H}$ be a hyperoval in $K$ defined by a function $g(u)$. 
Then the function $g_s(u)$ associated with the oval $\mathcal{O}_s$, whose nucleus is the origin, is equal to 
$$g_s(u)= \sum_{i=0}^{q}  a_i u^{i+1},$$
$$a_i = g(s) s^{(q-1)i/2-1} + g(s) \sum_{v\in S, v \ne s} g(v)\big(g(s)v+sg(v)\big)^{(q-1)i/2-1}. $$
\end{theorem}

\begin{proof} 
From Theorem \ref{thm:g-function} we have 
\begin{eqnarray*}
a_i  
& =  & 
\sum_{z\in {\mathcal O}_s} z^{(q-1)i/2 -1} \\
& =  & 
 \left( \frac{s}{g(s)}\right)^{(q-1)i/2 -1} + \sum_{v\in S, v \ne s} \left(\frac{v}{g(v)} + \frac{s}{g(s)}\right)^{(q-1)i/2 -1}  \\
& =  & 
\left( \frac{s}{g(s)}\right)^{(q-1)i/2 -1} + \sum_{v\in S, v \ne s} \left(\frac{g(s)v+sg(v)}{g(s)g(v)} \right)^{(q-1)i/2 -1}  \\
& =  & 
g(s) s^{(q-1)i/2-1} + g(s) \sum_{v\in S, v \ne s} g(v)\big(g(s)v+sg(v)\big)^{(q-1)i/2-1},  
\end{eqnarray*} 
which is our claim. 
\end{proof}

\begin{theorem}
\label{thm:f-neighbor} 
Let $\mathcal{H}$ be a hyperoval in $K$ defined by a function $g(u)$.  
Then the Niho bent function $f(x)$ associated with the oval $\mathcal{O}_s$, whose nucleus is the origin, is equal to 
$$ f_s(x)=\sum_{j=0}^{m-1} \sum_{i=0}^{q}   \left( \frac{g(s)^{2^j}}{s^{i(q-1)+2^j}}  + 
\sum_{v\in S, v\ne s } \frac{g(s)^{2^j} g(v)^{2^j}}{(g(s)v +sg(v))^{i(q-1)+2^j}} \right)x^{i(q-1)+2^j}.$$
\end{theorem}

\begin{proof} 
From Theorem \ref{thm:oval-f} we have 
\begin{eqnarray*}
f_s(x) 
& =  & 
\sum_{j=0}^{m-1} \sum_{i=0}^{q}  \left(\sum_{z\in {\mathcal O}} \frac{1}{z^{i(q-1)+2^j}} \right)x^{i(q-1)+2^j} \\
& =  & 
\sum_{j=0}^{m-1} \sum_{i=0}^{q}  \left( \left( \frac{s}{g(s)}\right)^{-i(q-1)-2^j} + 
\sum_{v\in S, v \ne s} \left( \frac{v}{g(v)} + \frac{s}{g(s)} \right)^{-i(q-1)-2^j} \right)x^{i(q-1)+2^j} \\ 
& =  & 
\sum_{j=0}^{m-1} \sum_{i=0}^{q}  \left( \left( \frac{g(s)}{s}\right)^{i(q-1)+2^j} + 
\sum_{v\in S, v \ne s} \left( \frac{g(s)g(v)}{g(s)v+sg(v)}  \right)^{i(q-1)+2^j} \right)x^{i(q-1)+2^j} \\
& =  & 
\sum_{j=0}^{m-1} \sum_{i=0}^{q}   \left( \frac{g(s)^{2^j}}{s^{i(q-1)+2^j} }+ 
\sum_{v\in S, v \ne s}  \frac{g(s)^{2^j}g(v)^{2^j}}{(g(s)v+sg(v))^{i(q-1)+2^j}}   \right)x^{i(q-1)+2^j},  
\end{eqnarray*} 
which completes the proof. 
\end{proof} 

Theorems \ref{thm:g-neighbor} and \ref{thm:f-neighbor} allow us to find all possible $g$-functions 
and Niho bent functions associated with a fixed hyperoval. 
For small values of $m$ Magma \cite{Bosma} can produce explicit list of such functions. 

\section{Equivalence classes of Niho bent functions}  
\label{Niho-ovals} 

In this section we describe the equivalence classes of Niho bent functions for all known hyperovals. 
If $h(t)$ is an o-polynomial then its corresponding hyperoval is 
\begin{eqnarray} 
\label{D(h)} 
{\mathcal D} (h) = \{  (t:h(t):1) \mid t \in F\} \cup \{ (0:1:0), (1:0:0) \}. 
\end{eqnarray}
Hyperoval ${\mathcal D} (h)$ contains all points of the fundamental quadrangle $\{ X, Y, Z, W\}$, where 
$Y=(1:0:0)$, $Z=(0:1:0)$, $X=(0:0:1)$ and $W=(1:1:1)$. Conversely, any hyperoval containing the 
 fundamental quadrangle can be written in the form (\ref{D(h)}). 
With this hyperoval we associate the oval 
$${\mathcal E} (h) = \{  (t:h(t):1) \mid t \in F\} \cup \{ (0:1:0) \},$$
whose nucleus is the point $Y=(1:0:0)$. 

To obtain each of the ovals contained in a given hyperoval $\mathcal{H}$, we choose a point of each orbit 
of the stabilizer of  $\mathcal{H}$, and map this point to the point $Y=(1:0:0)$, ensuring that the resulting image of 
$\mathcal{H}$ contains the fundamental quadrangle. The image of $\mathcal{H}$ can be written as 
$\mathcal{H}' = \mathcal{D}(h')$ and the corresponding oval is 
${\mathcal E} (h') = \{  (t:h'(t):1) \mid t \in F\} \cup \{ (0:1:0) \}$ and has nucleus $Y=(1:0:0)$ as required. 

Each permutation of the points $\{ X, Y, Z, W\}$ defines a collineation of $PG(2,q)$. These 24 maps were 
considered in \cite{Cher1988}. In particular, let us consider the maps \cite{OKeefe1990}, defined by 
\begin{eqnarray*}
\pi_1 \big( (a:b:c) \big) & = & (b:a:c),   \\ 
\pi_2 \big( (a:b:c) \big) & = & (c:b:a),  \\ 
\pi_3 \big( (a:b:c) \big) & = & (a: (a+b):(a+c)).  
 \end{eqnarray*} 
They map hyperoval ${\mathcal D} (h)$ to the equivalent hyperoval ${\mathcal D} (h_i)$, where 
\begin{eqnarray*} 
h_1(t) & = & h^{-1}(t),  \\
h_2(t) & = & t \ h(1/t),  \quad h(0)=0, \\ 
h_3(t) & = & t + (t+1) h(t/(t+1)), \quad  h(1)=1 .   
 \end{eqnarray*} 
Note that $\pi_1(Z) = Y$, $\pi_2(X) = Y$ and $\pi_3(W) = Y$. 
Therefore, $\pi_1$ maps the oval  ${\mathcal D} (h) \setminus \{ Z\}$ to the equivalent oval ${\mathcal E} (h_1)$, 
$\pi_2$ maps the oval  ${\mathcal D} (h) \setminus \{ X\}$ to the equivalent oval ${\mathcal E} (h_2)$ and 
$\pi_3$ maps the oval  ${\mathcal D} (h) \setminus \{ W\}$ to the equivalent oval ${\mathcal E} (h_3)$.

\subsection{Niho bent functions associated with the hyperconic}
\label{hyperconic}

In \cite{Leander2006} there were introduced Niho bent functions of the form
$$f_r(x)=Tr \left(a x^{2^m+1} + \sum_{i=1}^{2^{r-1}-1} x^{d_i} \right),$$
where $1<r<m$, $\gcd(r,m)=1$, $2^r d_i=(2^m-1)i+2^r$ and $a\in K$, $a+\bar{a}= 1$. 
It was shown in \cite{Buda2012} that $f_r(x)$ is associated with the translation hyperoval determined by the 
o-polynomial $h(t)=t^{2^{m-r}}$. Recall that the o-polynomials $h(t)=t^{2^r}$ and $h(t)=t^{2^{m-r}}$ determine 
equivalent hyperovals \cite{Hir}. 
To unify the notations, we introduce 
$$f_1(x)= Tr(a x^{2^m+1}).$$ 

\begin{proposition}
\label{prop:f-transl} 
Let $f_r(x) = tr(\lambda g_r(u))$, where $x=\lambda u$, $\lambda \in F$, $u\in S$.  Then 
$$g_1(u)=1,$$  
$$g_r(u) =  (u u^{2^{m-r}} + \bar{u}\bar{u}^{2^{m-r}}) / (u^{2^{m-r}} + \bar{u}^{2^{m-r}}), \quad g_r(1)=1  $$
for $r>1$. 
\end{proposition} 

\begin{proof} 
For $f_1(x)$ we clearly have $g_1(u)=1$. If $r>1$ then we have 

$$f_r(\lambda u)=Tr \left(a \lambda^2 +  \sum_{i=1}^{2^{r-1}-1} \lambda u^{d_i}\right)= 
tr \left(\lambda \left(1+ \sum_{i=1}^{2^{r-1}-1} u^{d_i}+ \sum_{i=1}^{2^{r-1}-1} \bar{u}^{d_i}\right)\right).$$
Therefore,  for $u\not= 1$ we have 
\begin{eqnarray*}
g_r(u)
&=& 
1+ \sum_{i=1}^{2^{r-1}-1} u^{d_i}+ \sum_{i=1}^{2^{r-1}-1} \bar{u}^{d_i} \\
&=& 
1+ \sum_{i=1}^{2^{r-1}-1} u^{-2i/2^r +1}+\sum_{i=1}^{2^{r-1}-1} \bar{u}^{-2i/2^r +1}\\
&=& 
1+ u^{-2/2^r +1}\cdot \frac{1-(u^{-2/2^r})^{2^{r-1}-1}}{1-u^{-2/2^r}} +
\bar{u}^{-2/2^r +1}\cdot \frac{1-(\bar{u}^{-2/2^r})^{2^{r-1}-1}}{1-\bar{u}^{-2/2^r}}  \\
&=& 
1+ \frac{u^{-2/2^r +1}-1}{1-u^{-2/2^r}} + \frac{\bar{u}^{-2/2^r +1}-1}{1-\bar{u}^{-2/2^r}} \\
&=& 
1+ \frac{\bar{u}^{2/2^r -1}-1}{1-\bar{u}^{2/2^r}} + \frac{u^{2/2^r -1}-1}{1- u^{2/2^r}} \\
&=& 
\frac{u + \bar{u} + u \bar{u}^{2/2^r} + \bar{u}u^{2/2^r}}{u^{2/2^r} + \bar{u}^{2/2^r}} \\
&=& 
\frac{(u^{2^{-r}} + \bar{u}^{2^{-r}} ) ( u \bar{u}^{2^{-r}} + \bar{u}u^{2^{-r}})} {u^{2^{1-r}} + \bar{u}^{2^{1-r}}} \\
&=& 
\frac{u \bar{u}^{2^{-r}} + \bar{u}u^{2^{-r}}} {u^{2^{-r}} + \bar{u}^{2^{-r}}} \\
&=& 
\frac{u u^{2^{m-r}} + \bar{u}\bar{u}^{2^{m-r}}} {u^{2^{m-r}} + \bar{u}^{2^{m-r}}} . 
\end{eqnarray*}
Clearly,   $g_r(1)=1$. 
\end{proof} 
 
In particular, 
$$g_{m-1} (u)= \frac{u u^2 + \bar{u}\bar{u}^2} {u^2 + \bar{u}^2} =  \frac{u^3 + \bar{u}^3} {u^2 + \bar{u}^2} = 
 \frac{1+ u^2 + \bar{u}^2} {u+ \bar{u}} =  \frac{1} {u+ \bar{u}} + u+ \bar{u}.$$

The hyperconic is given by the function $g(u)=1$: 
$$\mathcal{H} = \{ u \mid u \in S\} \cup \{ 0 \}. $$ 

\begin{theorem}
\label{thm:hyperconic} 
Let $m \ge 3$. Then there are exactly two equivalence classes of Niho bent functions associated with the hyperconic. 
Their representatives are the functions $f_1(x)$ and $f_{m-1}(x)$, and corresponding $g$-function representatives 
are $g_1(u)$ and $g_{m-1}(u)$. 
\end{theorem}

\begin{proof}   
Hyperconic is defined by the o-polynomial $h_0(t) = t^2$.  
The stabilizer of the hyperconic $\mathcal{H}={\mathcal D} (t^2)$ has two orbits 
\cite{Hir,OKeefe1990,OKeefe1994} on $\mathcal{H}$, 
one orbit contains the point $Y$, and the second one contains points $Z, X, W$. 
Therefore, there are exactly two inequivalent ovals $\mathcal{E}(h)$ with nucleus $Y$. 
They are defined by  the o-polynomials $h_0(t) = t^2$ and $h_1(t) = t^{1/2}$ (obtained by using the map $\pi_1$). 
Their $g$-functions by Theorem \ref{thm:oval} are  
\begin{eqnarray*}
g(u) 
&=& 
h_0^{-1} \left(\frac{ \langle i, u \rangle}{ \langle 1, u\rangle} \right) \langle 1, u\rangle = 
\left(\frac{ \langle i, u \rangle}{ \langle 1, u\rangle} \right)^{1/2} \langle 1, u\rangle \\
&=&  
(\bar{i}u + i \bar{u})^{1/2}(u+\bar{u})^{1/2} = (\bar{i} u^2 + iu^2 +1)^{1/2} =  \langle i^{1/2}, u \rangle +1, 
\end{eqnarray*} 
\begin{eqnarray*}
g'(u) 
&=& 
h_1^{-1} \left(\frac{ \langle i, u \rangle}{ \langle 1, u\rangle} \right) \langle 1, u\rangle = 
\left(\frac{ \langle i, u \rangle}{ \langle 1, u\rangle} \right)^{2} \langle 1, u\rangle \\
&=&  
\frac{\bar{i}^2u^2 + i^2 \bar{u}^2} {u+\bar{u}} = \frac{1} {u+ \bar{u}} + (\bar{i}^2 u + i^2u), 
\end{eqnarray*} 
which are equivalent to the functions $g_1(u)=1$ and $g_{m-1}(u) =  \frac{1} {u+ \bar{u}} + u+ \bar{u}$ respectively.  
Therefore, there are exactly two equivalence classes of Niho bent functions associated with the hyperconic and 
their representatives are functions $f_1(x)$ and $f_{m-1}(x)$. 
\end{proof}

 \subsection{Niho bent functions associated with the translation hyperovals}
 \label{translation}

The irregular translation hyperoval is defined by the o-polynomial $h(t)=t^{2^r}$, $1 < r <m-1$, $\gcd(r,m) \ne 1$. 
Recall that the o-polynomials $h(t)=t^{2^r}$ and $h(t)=t^{2^{m-r}}$ determine equivalent hyperovals \cite{Hir}. 

\begin{theorem}
\label{thm:translation} 
Let $m\ge 5$, $1 < r <m-1$, $\gcd(r,m) \ne 1$. 
There are exactly three equivalence classes of Niho bent functions associated with the translation hyperoval 
${\mathcal D} (t^{2^r})$. 
Their representatives are the functions $f_r(x)$,  $f_{m-r}(x)$ and  
$f(x) = tr \Big( \langle i, x \rangle^{(1-2^r)^{-1}}  { \langle 1, x\rangle}^{q-(1-2^r)^{-1}} \Big)$, 
where $(1-2^r)^{-1}$ is the inverse of $(1-2^r)$ modulo $q-1$. 
Corresponding $g$-function representatives are $g_r(u)$, $g_{m-r}(u)$ and 
$g(u) = \langle i, u \rangle^{(1-2^r)^{-1}}  { \langle 1, u\rangle}^{q-(1-2^r)^{-1}}$. 
\end{theorem}

\begin{proof}  
The stabilizer of the translation hyperoval $\mathcal{H}={\mathcal D} (t^{2^r})$ has three orbits 
\cite{Hir,OKeefe1990,OKeefe1994} on $\mathcal{H}$, one orbit contains the point $Y$, 
the second one contains the point $Z$, and the third one contains the points $X,W$. 
Therefore, there are exactly three inequivalent ovals $\mathcal{E}(h)$ with nucleus  $Y$. 
They are defined by the o-polynomials 
$h_0(t) = t^r$,  $h_1(t) = t^{2^{m-r}}$ (obtained by using the map $\pi_1$) and  
$h_2(t) = t^{1-2^r}$ (obtained by using the map $\pi_2$). 
Their $g$-functions by Theorem \ref{thm:oval} are  respectively 
\begin{eqnarray*}
g(u) 
&=& 
h_0^{-1} \left(\frac{ \langle i, u \rangle}{ \langle 1, u\rangle} \right) \langle 1, u\rangle \\
&=& 
\left(\frac{ \langle i, u \rangle}{ \langle 1, u\rangle} \right)^{2^{m-r}} \langle 1, u\rangle \\
&=&  
\frac{(\bar{i}u + i \bar{u})^{2^{m-r}}(u+\bar{u})}  {u^{2^{m-r}}+\bar{u}^{2^{m-r}} } \\
&=& 
\frac{uu^{2^{m-r}} + \bar{u}\bar{u}^{2^{m-r}}}  {u^{2^{m-r}}+\bar{u}^{2^{m-r}} } + 
(i^{2^{m-r}} u + \bar{i}^{2^{m-r}}\bar{u}) \\
&=& 
 g_r(u) + \langle \bar{i}^{2^{m-r}}, u \rangle , 
\end{eqnarray*} 
$$
g'(u) =  
h_1^{-1} \left(\frac{ \langle i, u \rangle}{ \langle 1, u\rangle} \right) \langle 1, u\rangle = 
\left(\frac{ \langle i, u \rangle}{ \langle 1, u\rangle} \right)^{2^r} \langle 1, u\rangle =
g_{m-r}(u) + \langle \bar{i}^{2^r}, u \rangle , 
$$
\begin{eqnarray*} 
g''(u) 
&=& 
h_2^{-1} \left(\frac{ \langle i, u \rangle}{ \langle 1, u\rangle} \right) \langle 1, u\rangle = 
\left(\frac{ \langle i, u \rangle}{ \langle 1, u\rangle} \right)^{(1-2^r)^{-1}} \langle 1, u\rangle \\
&=& 
\langle i, u \rangle^{(1-2^r)^{-1}}  { \langle 1, u\rangle}^{q-(1-2^r)^{-1}} .
\end{eqnarray*}
First two $g$-functions are equivalent to $g_r(u)$ and $g_{m-r}(u)$ respectively.  
\end{proof} 

 We note that  for the third function we have 
 $$(1-2^r)^{-1}  \equiv - \sum_{j=0}^{s-1} 2^{rj} \pmod{q-1} ,$$
 where $rs \equiv 1 \pmod{m}$. 

\begin{example}
\label{ex:translation} 
Let $m=5$.  Then there is only one (irregular) translation hyperoval (up to equivalency). 
It is defined by the o-polynomial $h(t)=t^4$.  
There are exactly three equivalence classes of Niho bent functions associated with this translation hyperoval.  
Their associated $g$-functions are 
$g_2(u)= 1 +  u^{16} + \bar{u}^{16}$, $g_3(u)= 1+ u^8 + u^9 + u^{16} + \bar{u}^8 + \bar{u}^9 + \bar{u}^{16}$, 
$g'(u)=1+\omega u^4 + \omega u^5 + \omega u^8 + \omega u^9 + \omega u^{12} + \omega u^{13} 
+ \bar{\omega} \bar{u}^4+ \bar{\omega}\bar{u}^5 + \bar{\omega}\bar{u}^8+\bar{\omega}\bar{u}^9 + 
\bar{\omega}\bar{u}^{12}+\bar{\omega} \bar{u}^{13}$. 
Hence corresponding Niho bent functions are $f_2(x)$,  $f_3(x)$ and 
$f'(x)= Tr ( \omega x^{528}+ \omega x^{466} + \omega x^{962} + \omega x^{404} + 
\omega x^{900} + \omega x^{342} + \omega x^{838} )$. 
\end{example}

We remind that if $g(u)=0$ then we assume that $u/g(u)=u_{\infty}$ is the element on infinity in direction $u$.  
We can ensure $g(u) \ne 0$ for all $u\in S$, by taking in place of $g(u)$ an equivalent function 
$g(u) + \langle c, u\rangle$, with appropriate $c\in K$, see  \cite{Ab2017}. 
Then all elements of the hyperoval $\mathcal{H} = \{ u/g(u) \mid u \in S\} \cup \{ 0 \}$ will be in $K$. 
In our current case we will see that if $g_r(u)=0$ for some $u$, then we can take the equivalent function 
$g(u) = g_r(u)+u+\bar{u}$. Functions  $g_r(u)$ and $g_r(u) +u+\bar{u}$ produce equivalent hyperovals.  
We  study now when the equations $g_r(u)=0$ and $g_r(u) +u+\bar{u}=0$ have solutions in $S$. 

First we consider  the following  
\begin{lemma}
\label{lem:gcd} 
Let $\gcd(m,r)=1$. Then $\gcd(2^m +1,2^r +1)$ (respectively, $\gcd(2^m +1,2^r -1)$) is equal to either 
$1$ or $3$. 
Moreover, $\gcd(2^m +1,2^r +1)=3$ if and only if $m$ is odd and $r$ is odd. 
In addition, $\gcd(2^m +1,2^r -1)=3$ if and only if $m$ is odd and $r$ is even. 
\end{lemma}

\begin{proof} 
Let $m>r$. Consider transformations of the form 
$$\gcd(2^m +1,2^r +1) = \gcd(2^m -2^r,2^r +1) = \gcd(2^{m-r} -1,2^r +1),$$ 
$$\gcd(2^m +1,2^r -1) = \gcd(2^m +2^r,2^r +1) = \gcd(2^{m-r} +1,2^r +1).$$ 
Using consecutively such transformations, we will finally get at the end of the chain 
$\gcd(2 +1,2 -1)=1$ or  $\gcd(2 +1,2 +1)=3$, 
since $\gcd(m,r)=1$. 

We note that 3 divides $\gcd(2^m +1,2^r +1)$ if and only if $m$ is odd and $r$ is odd, and 
3 divides $\gcd(2^m +1,2^r -1)$ if and only if $m$ is odd and $r$ is even. 
\end{proof}

\begin{lemma}
\label{lem:g} 
Let $1<r<m-1$, $\gcd(m,r)=1$, and  let $S= \langle w \rangle$. Then 
\begin{enumerate}
\item
The equation $g_r(u)=0$ has a solution in $S$ if and only if $m$ is odd and $r$ is even. 
In this case $u=w^{(q+1)/3}$ or $u=\bar{w}^{(q+1)/3}$. 
\item 
The equation $g_r(u) +u+\bar{u}=0$ has a solution in $S$ if and only if $m$ is odd and $r$ is odd.  
In this case $u=w^{(q+1)/3}$ or $u=\bar{w}^{(q+1)/3}$. 
\end{enumerate}
\end{lemma}

\begin{proof} 
1. We have  $g_r(u)=0$ if and only if  $u u^{2^{m-r}} + \bar{u}\bar{u}^{2^{m-r}} =0$  
if and only if $u u^{2^{m-r}} = u^{1+2^{m-r}} \in F$. 
Denoting $u=w^t \ne 1$,  we have $t(1+2^{m-r}) \equiv 0 \pmod{q+1}$. 
This congruence has a nontrivial solution if and only if  $\gcd(2^{m-r}+1, 2^m +1) \ne 1$.  
By Lemma \ref{lem:gcd} we see that  $\gcd(2^{m-r}+1, 2^m +1)=3$, which means that 
$m$ is odd and $r$ is even.  Therefore, $t= \pm (q+1)/3$. 

2. Since 
$$g_r(u) +u+\bar{u} = \frac{u u^{2^{m-r}} + \bar{u}\bar{u}^{2^{m-r}}} {u^{2^{m-r}} + \bar{u}^{2^{m-r}}} +u+\bar{u} = 
\frac{\bar{u} u^{2^{m-r}} +  u \bar{u}^{2^{m-r}}} {u^{2^{m-r}} + \bar{u}^{2^{m-r}}}, $$
we see that   $g_r(u) +u+\bar{u}=0$ if and only if  
$\bar{u} u^{2^{m-r}} +  u \bar{u}^{2^{m-r}} =0$  
if and only if $\bar{u} u^{2^{m-r}} = u^{2^{m-r}-1} \in F$. 
Denoting $u=w^t$, we have $t(2^{m-r}-1) \equiv 0 \pmod{q+1}$. 
This congruence has a nontrivial solution if and only if  $\gcd(2^{m-r}-1, 2^m +1) \ne 1$.  
By Lemma \ref{lem:gcd} we see that  $\gcd(2^{m-r}-1, 2^m +1)=3$, which means that 
$m$ is odd and $r$ is odd.  Therefore, $t= \pm (q+1)/3$. 
\end{proof}

We can calculate now the bent function $f_r(x)$ in other terms. 
Let $x=\lambda u$, $\lambda \in F$, $u\in S$. Then  $\lambda = \sqrt{x \bar{x}}$,  $u = \sqrt{x/ \bar{x}}$.  

If $x \in F$, then $u=1$ and 
$$f_r(x) = f_r(\lambda u) = tr (\lambda g_r(u)) = tr (\lambda) = tr ( \sqrt{x \bar{x}} ).$$

If $x\not\in F$, then $u\ne 1$ and 
\begin{eqnarray*}
f_r(x) = f_r(\lambda u) = tr (\lambda g_r(u)) 
& = &  
tr  \left( \sqrt{x \bar{x}} \cdot 
\frac{(\sqrt{x/ \bar{x}})^{2^{m-r}+1} + (\sqrt{\bar{x}/x})^{2^{m-r}+1}} 
{(\sqrt{x/ \bar{x}})^{2^{m-r}} + (\sqrt{\bar{x}/x})^{2^{m-r}}}\right) \\
& = &  
tr \left(  \frac{x x^{2^{m-r}} + \bar{x}\bar{x}^{2^{m-r}}} {x^{2^{m-r}} + \bar{x}^{2^{m-r}}}\right).
\end{eqnarray*}

Therefore, 
$$f_r(x) = 
\left\{ \begin{array} {ll}
tr ( \sqrt{x \bar{x}}),                                            & {\rm if} \ x\in F, \\ 
tr \left(  \frac{x x^{2^{m-r}} + \bar{x}\bar{x}^{2^{m-r}}} {x^{2^{m-r}} + \bar{x}^{2^{m-r}}}\right), & {\rm if} \ x\not\in F.
\end{array} \right. $$
In other words,
$$f_r(x) = tr \left(\sqrt{x \bar{x}} + \sqrt{x \bar{x}} (x + \bar{x})^{q^2-1}+ 
(x x^{2^{m-r}} + \bar{x}\bar{x}^{2^{m-r}}) (x^{2^{m-r}} + \bar{x}^{2^{m-r}})^{q^2-2} \right).$$

\subsection{Niho bent functions associated with the Segre hyperovals}
\label{segre}

The Segre hyperoval \cite{Segre62,Segre71} is defined by the o-polynomial $h(t)=t^6$, 
where $m \ge 5$ and $m$ is odd. Since $m$ is odd, 
in place of $i \in K$ we can take $\omega \in K$, $\omega^2 + \omega +1 =0$. 

Following  \cite{Buda2015}, we make use of Dickson polynomials. Dickson polynomial $D_s(x)$ is a 
permutation polynomial if $s$ is relatively prime with $2^n-1$, and since $D_s \circ  D_{s'} = D_{ss'}$, 
the inverse of $D_s$ is $D_{s'}$ where $s'$ is the inverse of $s$ modulo $2^n-1$. 
Denoting the inverse of $5$ modulo $2^n-1$ by $\frac{1}{5}$ we get $D_5^{-1} = D_{\frac{1}{5}}$. 
Note that $D_5(x) = x + x^3 + x^5$. 

\begin{theorem}
\label{thm:segre} 
If $m>5$, $m$ odd, then there are exactly four equivalence classes of Niho bent functions associated 
with the Segre hyperoval. 
Their $g$-function representatives are 
$(\bar{\omega}u + \omega \bar{u})^{1/6}(u+\bar{u})^{5/6}$, 
$(\bar{\omega}u + \omega \bar{u})^6(u+\bar{u})^{q-6}$, 
$(\bar{\omega}u + \omega \bar{u})^{-1/5}(u+\bar{u})^{6/5}$ and 
$$\left( D_{1/5} ((\omega u+ \bar{\omega} \bar{u}) ( u+\bar{u})^{q-2})\right)^{q^2-2} \langle 1, u\rangle .$$

If $m=5$, then there are exactly two equivalence classes of Niho bent functions associated 
with the Segre hyperoval. 
Their $g$-function representatives are 
$(\bar{\omega}u + \omega \bar{u})^{1/6}(u+\bar{u})^{5/6}$ and 
$$\left( D_{1/5} ((\omega u+ \bar{\omega} \bar{u}) ( u+\bar{u})^{q-2})\right)^{q^2-2} \langle 1, u\rangle  .$$
\end{theorem}

\begin{proof}  
If $m>5$ then the stabilizer of the Segre hyperoval $\mathcal{H} = \mathcal{D}(t^6)$ has four orbits 
\cite{OKeefe1990,OKeefe1994} on $\mathcal{H}$, each orbit contains one point from the set $\{ X, Y, Z, W \}$. 
Therefore, there are exactly four inequivalent ovals $\mathcal{E}(h)$ with nucleus  $Y$. 
They are defined by the o-polynomials $h_0(t) = t^6$,  $h_1(t) = t^{1/6}$,  $h_2(t) = t^{1-6}$, and 
$h_3(t) = t+ (t+1)(t/(t+1))^6 $ (obtained by using the maps $\pi_1$, $\pi_2$ and $\pi_3$). 
Their $g$-functions by Theorem \ref{thm:oval} are  respectively 
$$
g_0(u) =  
h_0^{-1} \left(\frac{ \langle \omega, u \rangle}{ \langle 1, u\rangle} \right) \langle 1, u\rangle = 
\left(\frac{ \langle \omega, u \rangle}{ \langle 1, u\rangle} \right)^{1/6} \langle 1, u\rangle =
(\bar{\omega}u + \omega \bar{u})^{1/6}(u+\bar{u})^{5/6} , 
$$
$$
g_1(u) =  
h_1^{-1} \left(\frac{ \langle \omega, u \rangle}{ \langle 1, u\rangle} \right) \langle 1, u\rangle = 
\left(\frac{ \langle \omega, u \rangle}{ \langle 1, u\rangle} \right)^6 \langle 1, u\rangle =
(\bar{\omega}u + \omega \bar{u})^6(u+\bar{u})^{q-6} , 
$$
$$
g_2(u) =    
h_2^{-1} \left(\frac{ \langle \omega, u \rangle}{ \langle 1, u\rangle} \right) \langle 1, u\rangle  =    
\left(\frac{ \langle \omega, u \rangle}{ \langle 1, u\rangle} \right)^{-1/5} \langle 1, u\rangle  =  
(\bar{\omega}u + \omega \bar{u})^{-1/5}(u+\bar{u})^{6/5} . 
$$
For calculation of $g_3(u)$ we need to find the inverse of the function $h_3(t)$. 
Let $h_3(t) = t+ (t+1)(t/(t+1))^6 =z$. Put $s=t+1$. 
Then $z=1+s+s ((s+1)/s)^6 = 1+ \frac{1}{s} + \frac{1}{s^3} + \frac{1}{s^5}  = 1 + D_5(\frac{1}{s}) $.  
Therefore,  $\frac{1}{s} = D_{1/5}(z+1)$ and $t= s+1 = (D_{1/5}(z+1))^{q^2-2} +1$. 
Hence, 
$$h_3^{-1} (t) =   (D_{1/5}(t+1))^{q^2-2} +1.$$
It follows that 
\begin{eqnarray*}
g_3(u) 
& = &   
h_3^{-1} \left(\frac{ \langle \omega, u \rangle}{ \langle 1, u\rangle} \right) \langle 1, u\rangle \\
& = &  
\left( D_{1/5}\left(\frac{ \langle \omega, u \rangle}{ \langle 1, u\rangle} + 1\right)\right)^{q^2-2} 
\langle 1, u\rangle + \langle 1, u\rangle \\
& = &  
\left( D_{1/5}\left(\frac{ \langle \bar{\omega}, u \rangle}{ \langle 1, u\rangle} \right)\right)^{q^2-2} 
\langle 1, u\rangle + \langle 1, u\rangle \\
& = &  
\left( D_{1/5} ((\omega u+ \bar{\omega} \bar{u}) ( u+\bar{u})^{q-2})\right)^{q^2-2} 
\langle 1, u\rangle + \langle 1, u\rangle .
\end{eqnarray*}

If $m=5$ then the stabilizer of the Segre hyperoval $\mathcal{D}(t^6)$ has two orbits \cite{OKeefe1990,OKeefe1994} 
on $\mathcal{D}(t^6)$, one orbit contains points  $X, Y, Z$, and the second one contains the point $W$. 
Therefore, there are exactly two inequivalent ovals $\mathcal{E}(h)$ with nucleus  $Y$. 
They are defined by  o-polynomials $h_0(t) = t^6$ and $h_3(t) = t+ (t+1)(t/(t+1))^6 $. 
\end{proof}

Note that $6^{-1}$ modulo $q-1$ is $(5q-4)/6$ and $5^{-1}$ modulo $q-1$ is $(3q^2-2)/5$. 

For $m=5$ in  \cite{OKeefe1990} it was mistakenly stated that the two classes of inequivalent ovals for the Segre 
hyperoval are represented by $\mathcal{E}(t^6)$ and $\mathcal{E}(t^{1/6}) = \mathcal{E}(t^{26})$. 
In fact, they are represented by $\mathcal{E}(t^6)$ and $\mathcal{E}(h_3(t))$. 
Their corresponding $g$-functions can be taken as 
$\omega u^4 +  \bar{\omega} \bar{u}^4+ u^5 +\bar{u}^5 + \omega u^8 + \bar{\omega} \bar{u}^8 + 
u^9 + \bar{u}^9 +\omega u^{12} + \bar{\omega} \bar{u}^{12} + u^{13} + u^{20}$ 
and  $1+ \omega u^9 + \bar{\omega}\bar{u}^9 + \bar{\omega} u^{12} +  \omega \bar{u}^{12}$. 
Ovals $\mathcal{E}(t^6)$, $\mathcal{E}(t^{1/6})$ and $\mathcal{E}(t^{1-6})$ are equivalent, 
and this fact was first observed in  \cite{Buda2015} implicitly in the language of Niho bent functions.

\subsection{Niho bent functions associated with the Glynn  hyperovals}
\label{glynn}

Let $m\ge 7$ be odd. Also let $\sigma =2^{(m+1)/2}$, and 
$$\gamma  = 
\left\{ \begin{array} {ll} 
2^k,                                            & {\rm if} \ m=4k-1, \\ 
2^{3k+1}, & {\rm if} \ m=4k+1.
\end{array} \right. $$
Hence $\gamma^4 \equiv \sigma^2 \equiv 2 \pmod{q-1}$. The Glynn hyperovals \cite{Glynn1983} are 
defined by the o-polynomials $h(t) = t^{3\sigma +4}$ and $h(t) = t^{\sigma + \gamma}$. 

\begin{theorem}
\label{thm:glynn} 
If $m\ge 7$, $m$ odd, then there are exactly four equivalence classes of Niho bent functions associated 
with the Glynn hyperoval  $\mathcal{D}(t^{3\sigma +4})$. 
Their $g$-function representatives are 
$(\bar{\omega}u + \omega \bar{u})^{3\sigma /2 - 2}(u+\bar{u})^{3-3\sigma /2 }$, 
$(\bar{\omega}u + \omega \bar{u})^{3\sigma +4}(u+\bar{u})^{{-3\sigma -3}}$, 
$(\bar{\omega}u + \omega \bar{u})^{(1-\sigma)/3}(u+\bar{u})^{(2+\sigma)/3}$ and 
$h_3^{-1} \left(\frac{ \langle \omega, u \rangle}{ \langle 1, u\rangle} \right) \langle 1, u\rangle $, where 
$h_3(t) = t+ (t+1)(t/(t+1))^{3\sigma +4} $. 

If $m\ge 9$, $m$ odd, then there are exactly four equivalence classes of Niho bent functions associated 
with the Glynn hyperoval  $\mathcal{D}(t^{\sigma +\gamma})$. 
Their $g$-function representatives are 
$(\bar{\omega}u + \omega \bar{u})^{-\gamma^{-1}+\sigma -\gamma +1}(u+\bar{u})^{\gamma^{-1}-\sigma +\gamma }$, 
$(\bar{\omega}u + \omega \bar{u})^{\sigma +\gamma}(u+\bar{u})^{{1-\sigma -\gamma}}$, 
$(\bar{\omega}u + \omega \bar{u})^{(-2\gamma^{-1} -\gamma +1)/3}(u+\bar{u})^{(2\gamma^{-1} +\gamma +2)/3}$ and 
$h_3^{-1} \left(\frac{ \langle \omega, u \rangle}{ \langle 1, u\rangle} \right) \langle 1, u\rangle $, where 
$h_3(t) = t+ (t+1)(t/(t+1))^{\sigma +\gamma} $. 

If $m=7$, then there are exactly two equivalence classes of Niho bent functions associated 
with the Glynn hyperoval  $\mathcal{D}(t^{\sigma +\gamma})$. 
Their $g$-function representatives are 
$(\bar{\omega}u + \omega \bar{u})^{-\gamma^{-1}+\sigma -\gamma +1}(u+\bar{u})^{\gamma^{-1}-\sigma +\gamma }$ 
 and  $h_3^{-1} \left(\frac{ \langle \omega, u \rangle}{ \langle 1, u\rangle} \right) \langle 1, u\rangle $. 
\end{theorem}

\begin{proof}  
The stabilizer of the Glynn hyperoval $\mathcal{H} = \mathcal{D}(t^{3\sigma +4})$ has four orbits 
\cite{OKeefe1990,OKeefe1994} on $\mathcal{H}$, each orbit contains one point from the set $\{ X, Y, Z, W \}$. 
Therefore, there are exactly four inequivalent ovals $\mathcal{E}(h)$ with nucleus $Y$. 
They are defined \cite{Glynn1983} by  o-polynomials 
$h_0(t) = t^{3\sigma +4}$,  
$h_1(t) = t^{(3\sigma +4)^{-1}} = t^{3\sigma /2 - 2} $,  
$h_2(t) = t^{1-(3\sigma +4)} $, and 
$h_3(t) = t+ (t+1)(t/(t+1))^{3\sigma +4} $ (obtained by using the maps $\pi_1$, $\pi_2$ and $\pi_3$). 
Their $g$-functions by Theorem \ref{thm:oval} are  respectively 
$$
g_0(u) =  
h_0^{-1} \left(\frac{ \langle \omega, u \rangle}{ \langle 1, u\rangle} \right) \langle 1, u\rangle = 
\left(\frac{ \langle \omega, u \rangle}{ \langle 1, u\rangle} \right)^{3\sigma /2 - 2} \langle 1, u\rangle =
(\bar{\omega}u + \omega \bar{u})^{3\sigma /2 - 2}(u+\bar{u})^{3-3\sigma /2 } , 
$$
$$
g_1(u) =  
h_1^{-1} \left(\frac{ \langle \omega, u \rangle}{ \langle 1, u\rangle} \right) \langle 1, u\rangle = 
\left(\frac{ \langle \omega, u \rangle}{ \langle 1, u\rangle} \right)^{3\sigma +4} \langle 1, u\rangle =
(\bar{\omega}u + \omega \bar{u})^{3\sigma +4}(u+\bar{u})^{{-3\sigma -3}} , 
$$
$$
g_2(u) =    
h_2^{-1} \left(\frac{ \langle \omega, u \rangle}{ \langle 1, u\rangle} \right) \langle 1, u\rangle  =    
\left(\frac{ \langle \omega, u \rangle}{ \langle 1, u\rangle} \right)^{(1-\sigma)/3} \langle 1, u\rangle  =  
(\bar{\omega}u + \omega \bar{u})^{(1-\sigma)/3}(u+\bar{u})^{(2+\sigma)/3} , 
$$
$$
g_3(u) =    
h_3^{-1} \left(\frac{ \langle \omega, u \rangle}{ \langle 1, u\rangle} \right) \langle 1, u\rangle   .
$$

For $m\ge 9$ the stabilizer of the Glynn hyperoval $\mathcal{H} = \mathcal{D}(t^{\sigma +\gamma})$ has four orbits 
\cite{OKeefe1990,OKeefe1994} on $\mathcal{H}$, each orbit contains one point from the set $\{ X, Y, Z, W \}$. 
Therefore, there are exactly four inequivalent ovals $\mathcal{E}(h)$ with nucleus $Y$. 
They are defined \cite{Glynn1983} by  o-polynomials 
$h_0(t) = t^{\sigma +\gamma}$,   
$h_1(t) = t^{(\sigma +\gamma)^{-1}} = t^{-\gamma^{-1}+\sigma -\gamma +1} $,  
$h_2(t) = t^{1-(\sigma +\gamma)} $, and 
$h_3(t) = t+ (t+1)(t/(t+1))^{\sigma +\gamma} $ (obtained by using the maps $\pi_1$, $\pi_2$ and $\pi_3$). 
Their $g$-functions by Theorem \ref{thm:oval} are  respectively 
\begin{eqnarray*}
g_0(u) 
&=&  
h_0^{-1} \left(\frac{ \langle \omega, u \rangle}{ \langle 1, u\rangle} \right) \langle 1, u\rangle = 
\left(\frac{ \langle \omega, u \rangle}{ \langle 1, u\rangle} \right)^{-\gamma^{-1}+\sigma -\gamma +1} \langle 1, u\rangle \\
&=& 
(\bar{\omega}u + \omega \bar{u})^{-\gamma^{-1}+\sigma -\gamma +1}(u+\bar{u})^{\gamma^{-1}-\sigma +\gamma } , 
\end{eqnarray*}
$$
g_1(u) =  
h_1^{-1} \left(\frac{ \langle \omega, u \rangle}{ \langle 1, u\rangle} \right) \langle 1, u\rangle = 
\left(\frac{ \langle \omega, u \rangle}{ \langle 1, u\rangle} \right)^{\sigma +\gamma} \langle 1, u\rangle =
(\bar{\omega}u + \omega \bar{u})^{\sigma +\gamma}(u+\bar{u})^{{1-\sigma -\gamma}} , 
$$
\begin{eqnarray*} 
g_2(u) 
&=&    
h_2^{-1} \left(\frac{ \langle \omega, u \rangle}{ \langle 1, u\rangle} \right) \langle 1, u\rangle  =    
\left(\frac{ \langle \omega, u \rangle}{ \langle 1, u\rangle} \right)^{(-2\gamma^{-1} -\gamma +1)/3} \langle 1, u\rangle \\
&=&  
(\bar{\omega}u + \omega \bar{u})^{(-2\gamma^{-1} -\gamma +1)/3}(u+\bar{u})^{(2\gamma^{-1} +\gamma +2)/3} , 
\end{eqnarray*}
$$
g_3(u) =    
h_3^{-1} \left(\frac{ \langle \omega, u \rangle}{ \langle 1, u\rangle} \right) \langle 1, u\rangle   .
$$

If $m=7$ then the stabilizer of the Glynn hyperoval $\mathcal{H} = \mathcal{D}(t^{\sigma +\gamma})$ 
has two orbits \cite{OKeefe1990,OKeefe1994} 
on $\mathcal{H}$, one orbit contains points  $X, Y, Z$, and the second one contains the point $W$. 
Therefore, there are exactly two inequivalent ovals $\mathcal{E}(h)$ with nucleus $Y$. 
They are defined by  o-polynomials $h_0(t) = t^{\sigma +\gamma}$ and $h_3(t) = t+ (t+1)(t/(t+1))^{\sigma +\gamma} $. 
\end{proof}

For $m=7$ in  \cite{OKeefe1990} it was mistakenly stated that the two classes of inequivalent ovals of the Glynn 
hyperoval $\mathcal{H} = \mathcal{D}(t^{\sigma +\gamma})$ are represented by $\mathcal{E}(t^{20})$ and 
$\mathcal{E}(t^{1/20}) = \mathcal{E}(t^{108})$. 
In fact, they are represented by $\mathcal{E}(t^{20})$ and $\mathcal{E}(h_3(t))$.

\subsection{Niho bent functions associated with the Payne, Cherowitzo, Subiaco and Adelaide  hyperovals}
\label{payne}

The orders of the automorphism groups of  the Payne, Cherowitzo, Subiaco and Adelaide  hyperovals 
are at most $10m$. 
Therefore, the action of the automorphism group on the points of the hyperoval has more than  
$2^m/(10m)$ orbits. 
Hence, the number of equivalence classes of  Niho bent functions associated with these  
hyperovals increases exponentially as the dimension of the underlying vector space grows. 

Nevertheless we can get some information on these equivalence classes. First of all, in case of odd $m$, 
there are some distinguished classes related to the points of the fundamental quadrangle. 

The Payne hyperoval \cite{Payne1985} is defined by the o-polynomial $h(t)= t^{1/6} + t^{1/2} + t^{5/6} $, 
where $m \ge 5$ and $m$ is odd. The stabilizer of the Payne hyperoval has order $2m$ and has about 
$2^m/(2m)$ orbits. These are $\{ Z \}$, $\{ W \}$, $\{ X, Y \}$,  and sets 
$$M_s = \{ (v^{s2^j}, h(v^{s2^j}), 1) \mid j=1, \dots, m\} \cup  \{ (1, h(v^{s2^j}),  v^{s2^j}) \mid j=1, \dots, m\}$$  
of size $2d$ where $d$ divides $m$ and $v$ is a primitive element of $F$ (see \cite{OKeefe1990,OKeefe1994,Thas1988}).  
So we can get $g$-functions using maps $\pi_1$ and $\pi_3$, and applying  Theorem \ref{thm:oval} 
(classes coming from the map $\pi_2$ and from the original o-polynomial are glued together since 
$t\cdot  h(1/t) = h(t)$), and additionally points from the orbits $M_s$ provide $g$-functions and Niho bent functions 
by using Theorems \ref{thm:oval-f}, \ref{thm:g-function}, \ref{thm:g-neighbor}, \ref{thm:f-neighbor}.  
In order to use these theorems we need a presentation of the Payne hyperoval in the affine plane $K=AG(2,q)$. 
The Payne hyperoval can be given \cite{Fisher2006} as 
\begin{equation}
\label{payne-u}
\{ u + u^3 + u^{-3} \mid u \in S \} \cup \{ 0 \}. 
\end{equation}
The stabilizer of the hyperoval is $Gal(K/\mathbb{F}_2)$. 
Then short orbits are $\{ 0\}$, $\{ 1\}$, $\{ \omega, \bar{\omega} \}$.  

The Cherowitzo hyperoval \cite{Cher1998} is defined by the o-polynomial 
$h(t)= t^{\sigma} + t^{\sigma+2} + t^{3\sigma +4}$, 
where $m \ge 5$,  $m$ is odd, and $\sigma = 2^{(m+1)/2}$, so $\sigma^2 \equiv 2 \pmod{q-1}$. 
The stabilizer of the Cherowitzo hyperoval has order $m$ and has about $2^m/m$ orbits. 
These are $\{ X\}$, $\{Y \}$, $\{ Z \}$, $\{ W \}$,  and sets (see \cite{Bayens,OKeefe1990,OKeefe1994,OKeefe1996}) 
$$M_s =  \{ ( v^{s2^j}, h(v^{s2^j}),1) \mid j=1, \dots, m\}.$$  

Note \cite{Carlet2011} that for the Payne o-polynomial  $h(t)= t^{1/6} + t^{1/2} + t^{5/6} $ the inverse  is 
$$h^{-1}(t) = (D_{1/5}(t))^6,$$  
and for the Cherowitzo o-polynomial $h(t)= t^{\sigma} + t^{\sigma+2} + t^{3\sigma +4}$, $\sigma = 2^{(m+1)/2}$, 
the inverse is 
$$h^{-1}(t) = t(t^{\sigma +1} + t^3 +t)^{\sigma/2 -1}.$$ 
These formulas allow us to find explicit expressions for  
$g(u)=h^{-1} \left(\frac{ \langle \omega, u \rangle}{ \langle 1, u\rangle} \right) \langle 1, u\rangle$.

The Adelaide hyperovals \cite{Cher2003} are given by o-polynomials
$$h(t) = \frac{T(b^k)}{T(b)}(t+1) + \frac{T((bt+b^q)^k)}{T(b)}(t+T(b)t^{1/2}+1)^{1-k} +t^{1/2}, $$
where $m$ even, $b\in S$, $b\not=1$ and $k=\pm \frac{q-1}{3}$.  
Description of the Adelaide hyperovals with the help of functions $g(u)$  looks much more simpler. 
The set
\begin{equation*}
\label{adelaide}
\mathcal{H} = \left\{\frac{u}{1+ u^{(q-1)/3} + \bar{u}^{(q-1)/3}} \mid \ u\in S \right\} \cup \{ 0\}
\end{equation*}
gives representation of the Adelaide hyperoval \cite{Ab2017,Ab2019} in $K$. 
The stabilizer of the hyperoval \cite{Ab2017,Ab2019,Payne2005} has order $2m$  and is equal to $Gal(K/\mathbb{F}_2)$. 

Note that the Adelaide hyperoval can be written  \cite{Fisher2006} in the form (\ref{payne-u}) as well, 
by taking even $m$.

The Subiaco \cite{CherPen1996} o-polynomial is given by  
$$h(t) = \frac{d^2t^4 + d^2(1+d+d^2)t^3 + d^2(1+d+d^2)t^2 +d^2t}{(t^2+dt+1)^2} +t^{1/2} $$
where $d\in F$, $tr(1/d)=1$, and $d\not\in \mathbb{F}_4$ for $m\equiv 2 \pmod{4}$. 
This o-polynomial gives rise to two inequivalent hyperovals when $m\equiv 2 \pmod{4}$ 
and to a unique hyperoval when $m\not\equiv 2 \pmod{4}$. 
The stabilizer group of  the Subiaco hyperoval was calculated in \cite{OKeefe1996,Payne1995}. 
There is other description \cite{Ab2017,Ab2019} of the Subiaco hyperovals by using functions $g(u)$ which looks much 
more simpler. 

Let $m\not\equiv 2 \pmod{4}$. Then $q +1 \not\equiv 0 \pmod{5}$ and the Subiaco hyperoval given by  
$$g(u)=1+ u^{5} + \bar{u}^{5}.$$ 
Its stabilizer has order $2m$ and is equal to $Gal(K/\mathbb{F}_2) = \langle \tau \rangle$, $\tau(x)=x^2$.

Let $m\equiv 2 \pmod{4}$. Then $q+1 \equiv 0 \pmod{5}$.  Let $S= \langle w\rangle$. 
Then we take 
$$g_i(u)=1+ w^i u^{5} + \bar{w}^i \bar{u}^{5},$$
where $i=0$ or $i=1$. Let $\tau (x) = x^2$, $v=w^{(q+1)/5}$ and $\varphi_v (x) = vx$.  

If the Subiaco hyperoval given by the function $g_0(u)=1+ u^{5} + \bar{u}^{5}$  
then its stabilizer has order $10m$ and is equal to the semidirect product 
$\langle \varphi_{v} \rangle \langle \tau \rangle$. 

If the Subiaco hyperoval given by $g_1(u)=1+ w u^{5} + \bar{w}\bar{u}^{5}$ 
then its stabilizer has order $5m/2$ and is equal to the cyclic group 
$\langle \varphi_{w^3}\tau^4 \rangle$.

 \subsection{Niho bent functions in small dimensions}
 \label{small}

In this section we consider Niho bent functions in dimensions up to $m=6$. Let $a$ be an element from $K$ 
with property $a+\bar{a}=1$.  Calculations are done with the help of Magma \cite{Bosma}. 

Let $m=1$. Then there is only one hyperoval (the hyperconic) and its automorphism group is transitive 
on the points of the hyperoval \cite{Hir}. 
Therefore, there is only one oval and only one Niho bent function up to equivalence, 
which is obtained from the function $g(u)=1$. Hence the bent function is equivalent to $f(x)=tr(x^3) =Tr(ax^3)$.

Let $m=2$. Then there is again only one hyperoval (the hyperconic) and its automorphism group is transitive on 
the points of the hyperoval \cite{Hir}. 
Therefore, there is only one oval and only one Niho bent function up to equivalence, they are obtained from 
the function $g(u)=1$, and the bent function is equivalent to  $f(x)=tr(x^{10}) =Tr(ax^{10})$.

Let $m=3$. Then there is only one hyperoval (the hyperconic) and its automorphism group has two orbits on the points 
of the hyperoval \cite{Hir}. 
Therefore, there are two ovals (consequently two Niho bent functions) up to equivalence, they are obtained from 
the functions $g(u)=1$ and  $g'(u) =1+u^4+\bar{u}^4$. 
Bent functions are $f(x)=tr(x^{36}) = Tr(ax^{36})$ and $f'(x)=tr(x^{36}+x^{22}+x^{50}) = Tr(ax^{36}+x^{22})$. 

Let $m=4$. Then there are two inequivalent hyperovals,  the hyperconic and the Lunelli-Sce hyperoval. 
The automorphism group of the hyperconic has two orbits on the points of the hyperoval \cite{Hir}. 
Therefore, there are two related ovals (consequently two Niho bent functions) up to equivalence, 
obtained from functions $g(u)=1$ and  $g'(u) = 1+u^4+\bar{u}^4+u^5+\bar{u}^5+u^8+\bar{u}^8$.  
The automorphism group of the Lunelli-Sce hyperoval is transitive on the points of the hyperoval \cite{Brown2000,Korch78}. 
Therefore, it determines only one oval (consequently one Niho bent function) with the function 
$g''(u)= 1+u^5 +\bar{u}^5$ 
(the Lunelli-Sce hyperoval is the first non-trivial member of the Subiaco \cite{Brown2000,CherPen1996} 
and the Adelaide families \cite{Cher2003}). 
Bent functions associated with the hyperconic are  $f(x)=tr(x^{136}) =Tr(ax^{136})$ and 
$f'(x)=tr(x^{136}+x^{106}+x^{226}+x^{76}+\bar{x}^{106}+\bar{x}^{226}+\bar{x}^{76}) =
Tr(ax^{136}+x^{106}+x^{226}+x^{76})$. 
Bent function associated with the Lunelli-Sce hyperoval is   
$f''(x)=tr(x^{136}+x^{226}+\bar{x}^{226}) =Tr(ax^{136}+x^{226})$.

\begin{table}[ht] 
\caption{Hyperovals in $AG(2,32)$ and associated $g$-functions  } 
\label{oval32}
\begin{center}
\begin{tabular}{|c|c|c|}
\hline
Hyperoval & function $g(u)$  & $|Aut|$ \\
 \hline
Hyperconic & 1 & 163680 \\
\hline
Translation &    $1+ u^{16} + \bar{u}^{16}$   & 4960  \\  
\hline 
Segre    &       $1+ \omega u^9 + \bar{\omega}\bar{u}^9 + \bar{\omega} u^{12} +  \omega \bar{u}^{12}$,    & 465 \\
       &        $\omega \in S$, $\omega^3 =1$, $\omega \neq 1$     &  \\     
\hline 
Subiaco & $1+ u^5 + \bar{u}^5 + u+\bar{u}$   &   10 \\
(Payne)  &      &  \\                  
\hline
Cherowitzo &    $u^5+ u^8 + u^9 + \omega u^{12} + \omega u^{13} + \omega u^{16} + $  &   5\\
   & $\bar{u}^5+ \bar{u}^8 +\bar{u}^9+\bar{\omega} \bar{u}^{12}+\bar{\omega}\bar{u}^{13}+\bar{\omega} \bar{u}^{13}$,&\\       
       &        $\omega \in S$, $\omega^3 =1$, $\omega \neq 1$  &     \\     
       \hline                               
O'Keefe-Penttila  & $1+ \varepsilon^{123} u^9 + \bar{\varepsilon }^{123}\bar{u}^9 + u^{12} +  \bar{u}^{12}$, & 3\\
       &        $\varepsilon \in K$, $\varepsilon^{10}+\varepsilon^6+\varepsilon^5+\varepsilon^3+\varepsilon^2+\varepsilon +1=0$    &   \\      
\hline             
\end{tabular}
\end{center}
\label{oval5}
\end{table}

\begin{table}[ht]
\caption{Hyperovals in $AG(2,64)$ and associated $g$-functions  } 
\label{oval64}
\begin{center}
\begin{tabular}{|c|c|c|}
\hline
Hyperoval & function $g(u)$  & $|Aut|$ \\
 \hline
Hyperconic & 1  &  1572480  \\
\hline 
Subiaco & $1+ u^5 + \bar{u}^5$   &   60 \\
\hline 
Subiaco & $1+ w u^5 + \bar{w}^5\bar{u}^5$,   &   15 \\
       &        $S= \langle w \rangle$     &  \\      \hline 
Adelaide    &       $1+ u^{21} + \bar{u}^{21}$   & 12  \\
\hline             
\end{tabular}
\end{center}
\label{default}
\end{table}

Let $m=5$. Then there are 6 hyperovals, they are listed in Table \ref{oval32}. 
The hyperconic (2 orbits), thanslation  (3 orbits) and 
Segre  (2 orbits) hyperovals were analyzed in the previous subsections. 
The Subiaco and Payne hyperovals are equivalent for $m=5$. 
The stabilizer of this hyperoval is generated by automorphism $\tau$, where $\tau (x) = x^2$. 
The orbits of the stabilizer are: $\{ 0\}$, $\{ 1\}$, $\{ \omega, \bar{\omega}\}$,  and three other 
sets each containing 10 elements. 
Short orbits have the following $g$-functions: 
$g_0= 1+ u +u^5 + \bar{u} +  \bar{u}^5$, 
$g_1(u) = 1+u^5 +u^8 + u^{12} + u^{13}  + \bar{u}^5 +\bar{u}^8 + \bar{u}^{12} + \bar{u}^{13}$,  
$g_{\omega}(u) =  u^4 + \omega u^5 +  \bar{\omega} u^9 + u^{12} + \bar{\omega} u^{16} + 
 \bar{u}^4 + \bar{\omega} \bar{u}^5 +  \omega \bar{u}^9 + \bar{u}^{12} + \omega \bar{u}^{16}$.

For the Cherowitzo hyperoval the orbits are: $\{ 0\}$, $\{ 1\}$, $\{ \omega\}$, $\{\bar{\omega}\}$ and six other sets 
with 5 elements. The stabilizer of the Cherowitzo hyperoval is generated by automorphism $\tau^2$. 
Short orbits have the following $g$-functions: 
$g_0= u^5+ u^8 + u^9 + \omega u^{12} + \omega u^{13} + \omega u^{16} + 
 \bar{u}^5+ \bar{u}^8 +\bar{u}^9+\bar{\omega} \bar{u}^{12}+\bar{\omega}\bar{u}^{13}+\bar{\omega} \bar{u}^{13}$, 
$g_1(u) = 1 + \bar{\omega}u^4 + \omega u^5 + u^8 + \bar{\omega}u^9 
+ \omega u^{12} + \omega u^{13} + \omega u^{16}  
+ \omega \bar{u}^4 + \bar{\omega} \bar{u}^5 + \bar{u}^8 + \omega \bar{u}^9 
+ \bar{\omega} \bar{u}^{12} + \bar{\omega} \bar{u}^{13} + \bar{\omega} \bar{u}^{16}$, 
$g_{\omega}(u) =  \omega u^4 + \omega u^8 + u^9 + \bar{\omega} u^{12} + \omega u^{13} +  \omega u^{16} 
+ \bar{\omega} \bar{u}^4 + \bar{\omega} \bar{u}^8 + \bar{u}^9 + \omega \bar{u}^{12} 
+ \bar{\omega} \bar{u}^{13} +  \bar{\omega} \bar{u}^{16}$, 
$g_{\bar{\omega}}(u) = 1 + \bar{\omega} u^4 + u^5 + \bar{\omega} u^8 
+ \bar{\omega} u^9 + \omega u^{12} + \bar{\omega} u^{13} + u^{16}   
+ \omega \bar{u}^4 + \bar{u}^5 + \omega \bar{u}^8 
+ \omega \bar{u}^9 + \bar{\omega} \bar{u}^{12} + \omega \bar{u}^{13} + \bar{u}^{16}$.

For the O'Keefe-Penttila hyperoval \cite{OKeefe1992} the orbits are: $\{ 0\}$ and 11 other sets each 
containing 3 elements. Stabilizer of the O'Keefe-Penttila hyperoval is generated by the automorphism 
$\varphi_{\omega} (x)= \omega x$.

Let $m=6$. Then there are 4 known hyperovals, they are listed in Table \ref{oval64}. 
The hyperconic was analyzed in the previous sections. 
The stabilizer of the Subiaco hyperoval defined by $g(u)=1+ u^5 + \bar{u}^5$ has 3 orbits, 
containing 1, 5 and 60 elements respectively. 
$g$-function for the orbit with 5 elements is 
$g_1(u) = 1 + u^4  + u^5 + u^9 + u^{13} + u^{17} + u^{21} + u^{24}  +  u^{25} + u^{29} + 
\bar{u}^4  + \bar{u}^5 + \bar{u}^9 + \bar{u}^{13} + \bar{u}^{17} + \bar{u}^{21} + \bar{u}^{24}  +  
\bar{u}^{25} + \bar{u}^{29}$.

The stabilizer of the Subiaco hyperoval defined by $g(u)=1+ w u^5 + \bar{w}^5\bar{u}^5$ has 6 orbits: 
one orbit is $\{ 0\}$, the second orbit contains 5 elements, and four other orbits each containing 15 elements. 

The stabilizer of the Adelaide hyperoval defined by $g(u)=1+ u^{21} + \bar{u}^{21}$ has 8 orbits, 
two orbits are $\{ 0\}$ and $\{ 1\}$ (which has corresponding function 
$g_1(u) =  1 +  u^4 + u^5 + u^9 + u^{12} + u^{13} + u^{16} + u^{17} + u^{20} + u^{24} + u^{25} +  u^{32}  
+  \bar{u}^4 + \bar{u}^5 + \bar{u}^9 + \bar{u}^{12} + \bar{u}^{13} + \bar{u}^{16} + \bar{u}^{17} + \bar{u}^{20} 
+ \bar{u}^{24} + \bar{u}^{25} +  \bar{u}^{32}$), 
one orbit contains 4 elements, and five orbits each containing 12 elements.

\section{Conclusion}
\label{conclusion}

We considered equivalence classes of Niho bent functions associated with hyperovals. 
For all known types of hyperovals we described the equivalence classes of the corresponding Niho bent functions.   
For some types of hyperovals  the number of equivalence classes of the associated Niho bent functions are 
at most 4. In general, the number of equivalence classes of  associated 
Niho bent functions increases exponentially as the dimension of the underlying vector space grows. 
The equivalence classes were considered in detail in small dimensions.

\bigskip

{\bf Acknowledgments}

\medskip

The author would like to thank Claude Carlet, Sihem Mesnager and Alexander Pott for valuable discussions. 
This work was supported by grant 31S1366.  


\end{document}